\documentclass[11pt]{article}
\usepackage{amssymb, amsfonts, amsthm, mathrsfs, amsmath}
\usepackage[all,cmtip]{xy}
\usepackage{tipa}

\theoremstyle{definition}
\newtheorem{thm}{Theorem}[section]
\newtheorem{dfn}[thm]{Definition}
\newtheorem{lemma}[thm]{Lemma}
\newtheorem{cor}[thm]{Corollary}
\newtheorem{prop}[thm]{Proposition}
\newtheorem{ex}[thm]{Example}
\newtheorem{rmk}[thm]{Remark}

\def\ann{\mathrm{Ann}}
\def\endo{\mathrm{End}}

\def\img{\mathrm{Im}}
\def\id{\mathrm{id}}

\def\tr{\mathrm{Tr}}

\def\ccc{\mathbb{C}}

\def\zz{\mathbb{Z}}
\def\rr{\mathbb{R}}

\def\cle{\mathcal{E}}

\def\cll{\mathcal{L}}

\def\clo{\mathcal{O}}

\def\pt{\partial}
\def\p{\partial}

\def\ud{\mathrm{d}}

\def\vol{\mathrm{vol}}

\def\spann{\mathrm{Span}}

\def\bea{\begin{eqnarray}}
\def\eea{\end{eqnarray}}

\def\<{{\langle}}
\def\>{{\rangle}}

\begin{document}

\title{The Geometry of Three-Forms on Symplectic Six-Manifolds}
\author{Teng Fei}

\date{}

\maketitle{}

\begin{abstract}
In this paper, we investigate the geometries associated with 3-forms of various orbital types on a symplectic 6-manifold. We show that there are extremely rich geometric structures attached to certain unstable 3-forms arising naturally from degeneration of Calabi-Yau structures, which in turn provides us a new perspective towards the SYZ conjecture. We give concrete examples and demonstrate that the limiting behavior of the Type IIA flow can be used to detect canonical geometric structures on symplectic manifolds. 
\end{abstract}

\section{Introduction}

The classification 3-forms in dimension six has been known for a long time \cite{reichel1907, gurevich1935b, gurevich1964, chan1998, bryant2006b}. A distinguished feature of 3-forms in dimension six is that there are two open (a.k.a. stable) orbits under the natural general linear group action. This fact has been utilized by many authors, such as Hitchin \cite{hitchin2000} and Donaldson-Lehmann \cite{donaldson2024,donaldson2024b,donaldson2024c}, to study various geometric and deformation problems in complex and Calabi-Yau geometry. 

In this paper, we systematically investigate various geometric structures associated with 3-forms on a 6-dimensional manifold, including those 3-forms from unstable orbits under the general linear group action. We introduce several natural integrability conditions on 3-forms and study the relationships among them. We examine in detail the geometry of certain unstable 3-forms arising naturally from degeneration of Calabi-Yau structures. In particular, we show that such manifolds possess a canonical Lagrangian foliation as a result of degeneration of complex structures. Moreover, we prove that the leaves of this foliation naturally admit a pair of dual Hessian structures, where the local potentials of the metric automatically solve a real Monge-Amp\`ere equation, hence demonstrating a close relationship with the Strominger-Yau-Zaslow conjecture \cite{strominger1996}. As explicit examples, we study special 3-forms on certain nilmanifolds and solvmanifolds, where we also apply the Type IIA flow to detect canonical geometric structures. 

This paper is organized as follows. In Section 2, we recall the classification of 3-forms in dimension 6 under both the general linear group and the symplectic group action. We introduce certain homogeneous equivariant polynomials to describe the orbits of 3-forms. Section 3 is devoted to integrability conditions on 3-forms. In particular, we introduce the notion of $F$-harmonic 3-forms as an ideal integrability condition. In Section 4, we study the geometry associated to an $F$-harmonic 3-form lying in the orbit $\clo_0^+$. Such 3-forms arises naturally from the degeneration of Calabi-Yau structures and they induce extremely rich geometric structures on the underlying symplectic manifolds, including a canonical Lagrangian foliation with leaves equipped with dual Hessian structures with an invariant volume form. In Section 5 we work with certain symplectic nilmanifolds and solvmanifolds and examine the space of special 3-forms on them. We also investigate the limiting behavior of the Type IIA flow of 3-forms and use it as a tool to detect canonical structures on such manifolds.\\

\noindent \textbf {Acknowledgment: }The author would like to thank Robert Bryant, Tristan Collins, Duong H. Phong, Li-Sheng Tseng and Xiangwen Zhang for valuable discussions. This work is generously supported by the Simons Collaboration Grant 830816.  

\section{Classification of three-forms}

Let $V$ be a real vector space of dimension $6$. We denote by $\bigwedge^3V^*$ the space of 3-form on $V$. The group $\mathrm{GL}(V)$ naturally acts on $\bigwedge^3V^*$. This action is ``stable'' in the sense that it has open orbits. In fact, all orbits of this action have been classified classically, see for example \cite{reichel1907, gurevich1935b, gurevich1964, chan1998, bryant2006b}.

To better describe these orbits, we introduce the following definitions.
\begin{dfn}
A 3-form $\varphi\in\bigwedge^3V^*$ is called \emph{stable} if it lies in an open orbit under the $\mathrm{GL}(V)$-action.
\end{dfn}
For any 3-form $\varphi\in\bigwedge^3V^*$, its kernel $\ker\varphi$ is defined to be
\[\ker\varphi=\{v\in V:\iota_v\varphi=0\}.\]
\begin{dfn}
A 3-form $\varphi$ is called \emph{non-degenerate} if $\ker\varphi=\{0\}$.
\end{dfn}
As $\varphi$ induces a non-degenerate 3-form $\bar\varphi\in\bigwedge^3(V/\ker\varphi)^*$, it is easy to see that the only possible dimensions for $\ker\varphi$ are $0$, $1$, $3$, and $6$.

For later use, we introduce several equivariant homogeneous polynomials on the space $\bigwedge^3V^*$:
\begin{enumerate}
\item $K$ is an equivariant homogeneous degree 2 polynomial of $\varphi$ such that $K(\varphi)\in\endo~V\otimes\bigwedge^6V^*$ for any $\varphi$. It is  defined by 
    \[K(\varphi)(v)=-\iota_v\varphi\wedge\varphi\in\bigwedge\nolimits^5V^*\cong V\otimes\bigwedge\nolimits^6V^*\] for any $v\in V$.
\item $F$ is an equivariant homogeneous degree 3 polynomial of $\varphi$ such that $F(\varphi)\in\bigwedge^3V^*\otimes\bigwedge^6V^*$ for any $\varphi$. It is defined by
    \[F(\varphi)(v_1,v_2,v_3)=-2\varphi(K(\varphi)(v_1),v_2,v_3)\] for any $v_1,v_2,v_3\in V$. Modulo tensoring the one dimensional space $\bigwedge^6V^*$, $F(\varphi)$ is also a 3-form. Therefore it makes sense to talk about $\ker F(\varphi)$. By this definition, we know that
    \[\ker F(\varphi)=K(\varphi)^{-1}(\ker\varphi\otimes\bigwedge\nolimits^6V^*).\]
\item $Q$ is an equivariant homogeneous degree 4 polynomial of $\varphi$ such that $Q(\varphi)\in(\bigwedge^6V^*)^{\otimes 2}$ for any $\varphi$. It is defined by
    \[Q(\varphi)=-\varphi\wedge F(\varphi).\]
\end{enumerate}
\begin{dfn}
An element of $(\bigwedge^6V^*)^{\otimes 2}$ is called \emph{positive} if it can be written as $\tau^2$ for $0\neq \tau\in\bigwedge^6V^*$. An element of $(\bigwedge^6V^*)^{\otimes 2}$ is called \emph{negative} if its negative is positive.
\end{dfn}

It has been known for a long time that the $\mathrm{GL}(V)$-action on $\bigwedge^3V^*$ has 6 distinct orbits. Two of them are stable, which are exactly the open sets defined by $Q(\varphi)<0$ and $Q(\varphi)>0$ respectively. These two open orbits will be denoted by $\clo_-$ and $\clo_+$ throughout this paper. The union of the other 4 orbits is the hypersurface $Q(\varphi)=0$. In fact, these 4 orbits can be distinguished by the dimension of their kernels. As a consequence, we shall denote these orbits by $\clo_0$, $\clo_1$, $\clo_3$, $\clo_6=\{0\}$, where the subscript labels the dimension of their kernel spaces.

Regarding these orbits, we have the following table:
\[\begin{tabular}{c|c|l}
\hline
Orbit & Dimesion & Normal Form \\
\hline
$\clo_-$ & 20 & $e^{135}-e^{146}-e^{236}-e^{245}$ \\
\hline
$\clo_+$ & 20 & $e^{123}+e^{456}$ \\
\hline
$\clo_0$ & 19 & $e^{146}+e^{236}+e^{245}$ \\
\hline
$\clo_1$ & 15 & $e^{135}+e^{245}$ \\
\hline
$\clo_3$ & 10 & $e^{135}$ \\
\hline
$\clo_6=\{0\}$ & 0 & 0 \\
\hline
\end{tabular}\]

\begin{dfn}
For $\varphi\in\bigwedge^3V^*$, the space $\ann~\varphi$ is defined to be
\[\ann~\varphi:=\{\alpha\in V^*:\alpha\wedge\varphi=0\}.\]
Dually we define
\[(\ann~\varphi)^\perp:=\{v\in V:\alpha(v)=0\textrm{ for any }\alpha\in\ann~\varphi\}.\]
\end{dfn}
Noticing that $K(\varphi)$ is a linear map from $V$ to $V\otimes\bigwedge^6V^*$ and that $V\cong V\otimes\bigwedge^6V^*$ canonically up to scaling, we may view both $\ker K(\varphi)$ and $\img~K(\varphi)$ as subspaces of $V$. Therefore associated with each $\varphi\in\bigwedge^3V^*$, there are four natural subspaces of $V$: $\ker\varphi$, $\ker K(\varphi)$, $\img~K(\varphi)$, $(\ann~\varphi)^\perp$ and one subspace of $V^*$: $\ann~\varphi$. Regarding dimensions of these vector spaces, it is easy to obtain the following table using the normal forms above.
\[\begin{tabular}{|c|c|c|c|c|}
\hline
Orbit & $\dim\ker\varphi$ & $\dim\ker K(\varphi)$ & $\dim\img~K(\varphi)$ & $\dim~(\ann~\varphi)^\perp$ \\
\hline
$\clo_-$ & 0 & 0 & 6 & 6 \\
\hline
$\clo_+$ & 0 & 0 & 6 & 6 \\
\hline
$\clo_0$ & 0 & 3 & 3 & 6 \\
\hline
$\clo_1$ & 1 & 5 & 1 & 5 \\
\hline
$\clo_3$ & 3 & 6 & 0 & 3 \\
\hline
$\clo_6=\{0\}$ & 6 & 6 & 0 & 0 \\
\hline
\end{tabular}\]
Moreover, one can check case by case that if any two spaces in the same line has the same dimension, then they are equal. Incorporating the fact that $F(\varphi)=0$ if and only if $\img~K(\varphi)\subset\ker\varphi$, we get the following stratification result:
\begin{prop}
The zeroes of the polynomials $K$, $F$, and $Q$ can be characterized as follows:
\begin{enumerate}
\item $\{\varphi:\varphi=0\}=\clo_6$,
\item $\{\varphi:K(\varphi)=0\}=\clo_3\coprod\clo_6$,
\item $\{\varphi:F(\varphi)=0\}=\clo_1\coprod\clo_3\coprod\clo_6$,
\item $\{\varphi:Q(\varphi)=0\}=\clo_0\coprod\clo_1\coprod\clo_3\coprod\clo_6$.
\end{enumerate}
\end{prop}
In addition, the polynomials $K$, $F$ and $Q$ satisfy the following equations:
\begin{prop}~
\begin{enumerate}
\item $K(\varphi)\circ K(\varphi)=\dfrac{\id_V}{4}\cdot Q(\varphi)\in\endo~V\otimes(\bigwedge^6V^*)^{\otimes2}$,
\item $K(F(\varphi))=-K(\varphi)\cdot Q(\varphi)\in\endo~V\otimes(\bigwedge^6V^*)^{\otimes3}$,
\item $F(F(\varphi))=-\varphi\cdot Q^2(\varphi)\in\bigwedge^3V^*\otimes(\bigwedge^6V^*)^{\otimes4}$.
\end{enumerate}
\end{prop}
\begin{proof}
As these identities are algebraic in nature, we only need to verify them on some open subset of $\bigwedge^3V^*$ such as $\clo_-$, where we can apply Hitchin \cite{hitchin2000}'s result below.
\end{proof}

Most work in the literature focuses on the orbit $\clo_-$ as it is related to complex and K\"ahler geometry. As explained in the pioneering work of Hitchin \cite{hitchin2000}, for any $\varphi\in\clo_-$, we have $K(\varphi)\circ K(\varphi)=\id_V\cdot\lambda(\varphi)$ for some $0>\lambda(\varphi)\in(\bigwedge^6V^*)^{\otimes2}$. Let $\sqrt{-\lambda(\varphi)}\in\bigwedge^6V^*$ be either of the two square roots of $-\lambda(\varphi)$, we have that
\[J(\varphi)=\frac{K(\varphi)}{\sqrt{-\lambda(\varphi)}}:V\to V\]
defines a complex structure on $V$, such that the complex valued 3-form $\varphi+\sqrt{-1}\hat\varphi$, where $\hat\varphi=J(\varphi)^*\varphi$, is a nowhere vanishing $(3,0)$-form with respect to $J(\varphi)$. In this setting, we have
\bea\label{relations}
\begin{split}K(\varphi)&=J(\varphi)\cdot\sqrt{-\lambda(\varphi)},\\
F(\varphi)&=2\hat\varphi\cdot\sqrt{-\lambda(\varphi)},\\
Q(\varphi)&=4\lambda(\varphi).\end{split}
\eea
There is a parallel story about the orbit $\clo_+$, which is a ``para-complex'' or ``split'' version of Hitchin's theory, see for example \cite{etayo2006, fei2015b, hamilton2023}.

The cubic polynomial map $F:\bigwedge^3V^*\to\bigwedge^3V^*\otimes(\bigwedge^6V^*)^{\otimes 2}$ is $\mathrm{GL}(V)$-equivariant, therefore it maps orbits to orbits. It is not hard to deduce the following statement.
\begin{prop}\label{Fmap}
The cubic polynomial $F$ satisfies the following:
\begin{enumerate}
\item $F:\clo_\pm\to\clo_{\pm}\otimes(\bigwedge^6V^*)^{\otimes 2}$ is bijective.
\item $F(\clo_0)=\clo_3\otimes(\bigwedge^6V^*)^{\otimes 2}$.
\item $F=0$ on degenerate orbits $\clo_1$, $\clo_3$ and $\clo_6=\{0\}$.
\end{enumerate}
\end{prop}
It follows that the nontrivial part of $F$ is concentrated in the orbit $\clo_0$. Understanding the map $F:\clo_0\to\clo_3\otimes (\bigwedge^6V^*)^{\otimes 2}$ is equivalent to understanding the stabilizers of $\varphi\in\clo_0$ and $F(\varphi)\in\clo_3\otimes (\bigwedge^6V^*)^{\otimes 2}$ respectively.

For the simplicity of calculation, let us take $\varphi$ to be the normal form in our previous table. In addition, let us rename the vectors in $V^*$ by
\[\ud x^j=e^{2j-1},\quad \ud y^j=e^{2j},\quad\textrm{for }j=1,2,3.\]\
In this set-up, straightforward calculation yields
\[\begin{split}\varphi&=\ud x^1\wedge\ud y^2\wedge\ud y^3+\ud x^2\wedge\ud y^3\wedge\ud y^1+\ud x^3\wedge\ud y^1\wedge\ud y^2,\\
F(\varphi)&=4\ud y^1\wedge\ud y^2\wedge\ud y^3\otimes(\ud x^1\wedge\ud y^1\wedge\ud x^2\wedge\ud y^2\wedge\ud x^3\wedge\ud y^3).
\end{split}\]
Fix the basis $\{\ud x^1,\ud x^2,\ud x^3,\ud y^1,\ud y^2,\ud y^3\}$ of $V^*$ which we shall abbreviate as $\{\ud x,\ud y\}$. We have the following descriptions of the stabilizers of $F(\varphi)$ and $\varphi$.
\begin{prop}~
\begin{enumerate}
\item The stabilizer of $F(\varphi)$ consists of elements taking the matrix form
\[\begin{pmatrix}A & 0\\ B & C\end{pmatrix},\]
where $B$ is arbitrary, $A$ and $C$ are invertible $3\times3$ matrices satisfying $\det A\cdot\det^2C=1$.
\item The stabilizer of $\varphi$ is a subgroup of the stabilizer of $F(\varphi)$ whose elements further satisfy
\[A=\frac{C}{\det C},\textrm{ and }\tr(BC^{-1})=0.\]
\end{enumerate}
We also find that
\[K(\varphi)(\pt_{x^j})=0,\quad K(\varphi)(\pt_{y^j})=-2\pt_{x^j}\otimes(\ud x^1\wedge\ud y^1\wedge\ud x^2\wedge\ud y^2\wedge\ud x^3\wedge\ud y^3).\]
Consequently we have $\ker K(\varphi)=\img~K(\varphi)=\ker F(\varphi)=\spann\{\pt x\}.$
\end{prop}

People are also interested in the situation that $V$ is equipped with a symplectic form $\omega\in\bigwedge^2V^*$, in which case we shall consider the action of the group $\mathrm{Sp}(V,\omega)$ on the space of 3-forms. The Lefschetz decomposition
\[\bigwedge\nolimits^3V^*=\bigwedge\nolimits^3_0V^*\oplus\omega\wedge V^*\]
breaks the space of 3-forms on $V$ into $\mathrm{Sp}(V,\omega)$-irreducible pieces, where $\bigwedge^3_0V^*$ denotes the space of $\omega$-primitive 3-forms which can be characterized by
\[\bigwedge\nolimits^3_0V^*=\{\varphi:\Lambda\varphi=0\}=\{\varphi:\omega\wedge\varphi=0\},\]
where $\Lambda$ denotes the Lefschetz operator of contraction with $\omega$.
 
Since any $\varphi\in\omega\wedge V^*$ is automatically a member of the orbit $\clo_1$, it is natural to only consider the $\mathrm{Sp}(V,\omega)$-orbits in $\bigwedge^3_0V^*$. In particular, for any $\mathrm{GL}(V)$-orbit $\clo$, we would like to know how $\clo\cap\bigwedge^3_0V^*$ breaks into distinct $\mathrm{Sp}(V,\omega)$-orbits. This problem has been studied by many authors such as \cite{lychagin1983, banos2003, bryant2006b}. We now know that $\clo_-\cap\bigwedge^3_0V^*$ splits into two families of $\mathrm{Sp}(V,\omega)$-orbits $\clo_-^\pm(\mu)$, each parametrized by a positive scaling factor $\mu$, and $\clo_+\cap\bigwedge^3_0V^*$ decomposes as a family of $\mathrm{Sp}(V,\omega)$-orbits $\clo_+(\mu)$ parametrized by a positive scalar $\mu$, and that $\clo_j\cap\bigwedge^3_0V^*$ breaks as the union of two $\mathrm{Sp}(V,\omega)$-orbits $\clo_j^\pm$ for $j=0,1$, and that $\clo_3\cap\bigwedge^3_0V^*$ and $\clo_6=\{0\}$ are single $\mathrm{Sp}(V,\omega)$-orbits. Taken from \cite{bryant2006b}, the normal forms of these orbits can be summarized as follows:
\[\begin{tabular}{|c|c|l|}
\hline
Orbit & Dimension & Normal Form \\
\hline
$\clo_-^\pm(\mu)$ & $13$ & $\mu(e^{135}-e^{146}\mp e^{236}\mp e^{245})$ \\
\hline
$\clo_+(\mu)$ & $13$ & $\mu(e^{135}+e^{246})$ \\
\hline
$\clo_0^\pm$ & $13$ & $e^{146}\pm e^{236}\pm e^{245}$ \\
\hline
$\clo_1^\pm$ & $10$ & $(e^{13}\mp e^{24})\wedge e^5$ \\
\hline
$\clo_3\cap\bigwedge^3_0V^*$ & $7$ & $e^{135}$ \\
\hline
$\clo_6=\{0\}$ & 0 & $0$\\
\hline
\end{tabular}\]
In all these cases,  we have $\omega=e^{12}+e^{34}+e^{56}$.

An immediate consequence is the following important observation:
\begin{prop}
For any $\varphi\in\clo_0^-\coprod\clo_0^+=\clo_0\cap\bigwedge^3_0V^*$, we have that $\ker K(\varphi)=\img~K(\varphi)=\ker F(\varphi)$ is a Lagrangian subspace of $(V,\omega)$ such that both the restrictions of $\varphi$ and $F(\varphi)$ on it vanish. 
\end{prop}

Under the presence of the symplectic form $\omega$, the line $\bigwedge^6V^*$ is trivialized by $\dfrac{\omega^3}{3!}$, therefore we may canonically identify $K(\varphi)$, $F(\varphi)$ and $Q(\varphi)$ as elements in $\endo~V$, $\bigwedge_0^3V^*$ and $\rr$ respectively for any primitive 3-form $\varphi$. Moreover, we have the following statement parallel to Proposition \ref{Fmap}.
\begin{prop}
The homogeneous cubic polynomial $F:\bigwedge_0^3V^*\to\bigwedge_0^3V^*$ satisfies
\begin{enumerate}
\item $F:\clo_-^\pm(\mu)\to\clo_-^\pm(4\mu^3)$ is bijective.
\item $F:\clo_+(\mu)\to\clo_+(2\mu^3)$ is bijective.
\item $F:\clo_0^\pm\to\clo_3\cap\bigwedge^3_0V^*$ is surjective.
\item $F\equiv0$ on $\clo_1^\pm$, $\clo_3\cap\bigwedge^3_0V^*$, and $\clo_6=\{0\}$.
\end{enumerate}
\end{prop}

By using $\omega$, there are a few ways to construct a symmetric bilinear form that is quadratic in $\varphi$. It is not surprising that they all yield the same quadratic form as detailed in the following proposition.
\begin{prop}
There is a natural symmetric bilinear form $q(\omega,\varphi)$ on $V$ associated to the pair $(\omega,\varphi)$, which satisfies
\[q(\omega,\varphi)(v_1,v_2)=\omega(v_1,K(\varphi)(v_2))=\frac{\iota_{v_1}\varphi\wedge\iota_{v_2}\varphi\wedge\omega}{\omega^3/3!}=-\omega(\iota_{v_1} \varphi,\iota_{v_2}\varphi).\]
\end{prop}
\begin{proof}
Since everything is algebraic and the union $\coprod_{\mu>0}\clo_-^+(\mu)$ is an open subset in $\bigwedge^3_0V^*$, we only need to prove this proposition for every member $\varphi\in\clo_-^+(\mu)$. This special case has been done in \cite[Lemma 5]{fei2021b}.
\end{proof}
The symmetric bilinear form $q(\omega,\varphi)$ has different signature depending on the orbit type of $\varphi$. To be more specific, we have:
\begin{prop}\label{cases}~
\begin{enumerate}
\item For $\varphi\in\clo_-^+(\mu)$, the symmetric bilinear form $q(\omega,\varphi)$ is positive definite, namely its signature is $(0,6,0)$.
\item For $\varphi\in\clo_-^-(\mu)$, the symmetric bilinear form $q(\omega,\varphi)$ has signature $(0,2,4)$.
\item For $\varphi\in\clo_+(\mu)$, the symmetric bilinear form $q(\omega,\varphi)$ has split signature, namely its signature is $(0,3,3)$.
\item For $\varphi\in\clo_0^+$, the symmetric bilinear form $q(\omega,\varphi)$ has signature $(3,3,0)$. In particular the Lagrangian subspace $L:=\ker K(\varphi)=\img~K(\varphi)=\ker F(\varphi)$ is $q$-orthogonal to every vector in $V$. The induced bilinear form on $V/L$ is positive definite. As the map $K(\varphi):V/L\to L$ is an isomorphism, we also obtain a canonical positive definite metric on $L$.
\item For $\varphi\in\clo_0^-$, the symmetric bilinear form $q(\omega,\varphi)$ has signature $(3,1,2)$. In particular the Lagrangian subspace $L:=\ker K(\varphi)=\img~K(\varphi)=\ker F(\varphi)$ is $q$-orthogonal to every vector in $V$. The induced bilinear form on $V/L$ has signature $(1,2)$. As the map $K(\varphi):V/L\to L$ is an isomorphism, we also obtain a canonical metric of signature $(1,2)$ on $L$.
\item For $\varphi\in\clo_1^+$, the symmetric bilinear form $q(\omega,\varphi)$ has signature $(5,1,0)$. In particular the coisotropic subspace $\ker\varphi=\ker K(\varphi)$ is $q$-orthogonal to every vector in $V$. The induced bilinear form on $V/\ker\varphi$ is positive definite.
\item For $\varphi\in\clo_1^-$, the symmetric bilinear form $q(\omega,\varphi)$ has signature $(5,0,1)$. In particular the coisotropic subspace $\ker\varphi=\ker K(\varphi)$ is $q$-orthogonal to every vector in $V$. The induced bilinear form on $V/\ker\varphi$ is negative definite.
\item For $\varphi\in\clo_3\cap\bigwedge^3_0V^*$ or $\varphi=0$, the bilinear form $q(\omega,\varphi)$ vanishes. 
\end{enumerate}
\end{prop}

In the work \cite{fei2021b}, a primitive 3-form $\varphi$ is called positive if it is a member of the orbit $\clo_-^+(\mu)$ for some $\mu>0$. In this setting, the natural almost complex structure $J(\varphi)$ from Hitchin's construction is compatible with $\omega$ so the associated metric $g(\varphi)(\cdot,\cdot)=\omega(\cdot,J(\varphi)\cdot)$ is positive definite. We can define the norm of $\varphi$ through
\[|\varphi|^2\frac{\omega^3}{3!}=\varphi\wedge\hat\varphi,\]
which turns out to be the same as the norm of $\varphi$ measured under the metric $g(\varphi)$. In this set-up, the relations in (\ref{relations}) can be rewritten as \bea
\begin{split}K(\varphi)&=\frac{1}{2}|\varphi|^2J(\varphi),\\
F(\varphi)&=|\varphi|^2\hat\varphi,\\
Q(\varphi)&=-|\varphi|^4,\\
\mu&=\frac{1}{2}|\varphi|,\\
q(\omega,\varphi)&=\frac{|\varphi|^2}{2}g(\varphi).
\end{split}
\eea

\section{Integrability conditions}

From now on we shall always work under the presence of a symplectic form. Let $(M,\omega)$ be a connected symplectic 6-manifold, and let $\varphi$ be a smooth primitive 3-form on $M$. As explained in the previous section, we naturally have $K(\varphi)$, $F(\varphi)$ and $Q(\varphi)$ as smooth sections of $\endo~TM$, $\bigwedge^3_0T^*M$ and $\bigwedge^0T^*M$ respectively. Therefore we can talk about various notions of integrability associated to $\varphi$. For the convenience of notation, for $K(\varphi)$, $F(\varphi)$ and $Q(\varphi)$ and similar quantities, we drop the reference to $\varphi$ when there is no confusion.
\begin{dfn}~
\begin{enumerate}
\item The 3-form $\varphi$ is called \emph{integrable} if $\ud\varphi=0$. 
\item The 3-form $\varphi$ is called \emph{$K$-integrable} if the Nijenhuis tensor $N_K$ of $K(\varphi)$ vanishes. Here $N_K$ is defined by
\[N_K(X,Y)=-K^2[X,Y]+K([KX,Y]+[X,KY])-[KX,KY]\]
for any vector fields $X$ and $Y$ on $M$.
\item The 3-form $\varphi$ is called \emph{$F$-integrable} if $\ud F(\varphi)=0$.
\item The 3-form $\varphi$ is called \emph{$Q$-integrable} if $\ud Q(\varphi)=0$, namely $Q$ is a constant function on $M$.
\end{enumerate}
\end{dfn}
It is not surprising that these notions of integrability are related to each other. To explore this aspect, let us first establish a series of useful identities.

\begin{lemma}
For any vector fields $X$ and $Y$, we have
\bea
&&\iota_X\varphi\wedge F(\varphi)=-\varphi\wedge\iota_XF(\varphi)=\frac{1}{2}\iota_X(\varphi\wedge F(\varphi)),\label{o1}\\
&&\iota_X\varphi\wedge\iota_YF(\varphi)+\iota_Y\varphi\wedge\iota_XF(\varphi)=0,\label{o21}\\
&&\iota_Y\iota_X\varphi\wedge F(\varphi)=\varphi\wedge\iota_Y\iota_XF(\varphi)\label{o22}.
\eea
\end{lemma}
\begin{proof}
We shall prove this lemma by direct calculation. It is noteworthy that we have the following equations from the definition of $K$, $F$, and $Q$:
\bea
\iota_{KX}\frac{\omega^3}{3!}&=&-\iota_X\varphi\wedge\varphi,\label{K}\\
\iota_X F(\varphi)&=&-2\iota_{KX}\varphi,\label{F}\\
\varphi\wedge F(\varphi)&=&-Q(\varphi)\frac{\omega^3}{3!}\label{Q},\\
K^2&=&\frac{Q(\varphi)}{4}\id\label{square}.
\eea
Using (\ref{F}) and (\ref{K}), we get
\[-\varphi\wedge\iota_XF(\varphi)=2\varphi\wedge\iota_{KX}\varphi=-2\iota_{K^2X}\frac{\omega^3}{3!}=-\frac{Q}{2}\iota_X\frac{\omega^3}{3!}.\]
Therefore
\[\iota_X\varphi\wedge F(\varphi)-\varphi\wedge\iota_XF(\varphi)=\iota_X(\varphi\wedge F(\varphi))=-\iota_XQ\frac{\omega^3}{3!}=-2\varphi\wedge\iota_XF(\varphi)\]
and we prove (\ref{o1}).

Take $\iota_Y$ of left and right sides of (\ref{o1}), we get
\[\frac{1}{2}\iota_Y\iota_X(\varphi\wedge F(\varphi))=\iota_Y(\iota_X\varphi\wedge F(\varphi))=\iota_Y\iota_X\varphi\wedge F(\varphi)+\iota_X\varphi\wedge\iota_YF(\varphi).\]
As this expression is anti-symmetric in $X$ and $Y$, we obtain (\ref{o21}).

Take $\iota_Y$ of left and middle sides of (\ref{o1}), we get
\[\iota_Y\iota_X\varphi\wedge F(\varphi)+\iota_X\varphi\wedge\iota_YF(\varphi)=-\iota_Y\varphi\wedge\iota_XF(\varphi)+\varphi\wedge\iota_Y\iota_X F(\varphi).\]
Plug in (\ref{o21}) we prove (\ref{o22}).
\end{proof}
Next we establish an identity relating the Nijenhuis tensor $N_K$ with other quantities like $\ud\varphi$ and $\ud F(\varphi)$.
\begin{thm}
The Nijenhuis tensor $N_K$ satisfies
\bea\label{id}
\iota_{N_K(X,Y)}\frac{\omega^3}{3!}&=&\iota_Y\iota_X\ud\varphi\wedge F(\varphi)-\ud\varphi\wedge\iota_Y\iota_X F(\varphi)+2\varphi\wedge(\iota_Y\iota_{KX}-\iota_X\iota_{KY})\ud\varphi\nonumber\\
&&+\varphi\wedge\iota_Y\iota_X\ud F(\varphi).
\eea
\end{thm} 
\begin{proof}
From (\ref{K}) and (\ref{F}) we know
\bea
&&-\iota_{[KX,KY]}\frac{\omega^3}{3!}=-\iota_{\cll_{KX}(KY)}\frac{\omega^3}{3!}\nonumber\\
&=&-\cll_{KX}\iota_{KY}\frac{\omega^3}{3!}+\iota_{KY}\cll_{KX}\frac{\omega^3}{3!} \nonumber\\
&=&-\cll_{KX}\iota_{KY}\frac{\omega^3}{3!} +\cll_{KY}\iota_{KX}\frac{\omega^3}{3!}-d\iota_{KY}\iota_{KX}\frac{\omega^3}{3!}\nonumber\\
&=&\cll_{KX}(\iota_Y\varphi\wedge\varphi)-\cll_{KY}(\iota_X\varphi\wedge\varphi)-d\iota_{KY}\iota_{KX}\frac{\omega^3}{3!}\nonumber\\
&=&\left(\cll_{KX}\iota_Y\varphi-\cll_{KY}\iota_X\varphi\right)\wedge\varphi-d\iota_{KY}\iota_{KX}\frac{\omega^3}{3!} +\iota_Y\varphi\wedge\iota_{KX}\ud\varphi\nonumber\\
&&-\iota_X\varphi\wedge\iota_{KY}\ud\varphi+\frac{1}{2}\iota_X\varphi\wedge\ud\iota_YF(\varphi)-\frac{1}{2}\iota_Y\varphi\wedge\ud\iota_XF(\varphi).\label{t1}
\eea
Similarly, we have that
\bea
&&\iota_{K([KX,Y]+[X,KY])}\frac{\omega^3}{3!}=-\iota_{[KX,Y]}\varphi\wedge\varphi-(X\leftrightarrow Y)\nonumber\\
&=&(-\cll_{KX}\iota_Y\varphi+\iota_Y\cll_{KX}\varphi)\wedge\varphi-(X\leftrightarrow Y)\nonumber\\
&=&\left(-\cll_{KX}\iota_Y\varphi+\iota_Y\iota_{KX}\ud\varphi-\frac{1}{2}\iota_Y\ud\iota_XF(\varphi)\right)\wedge\varphi-(X\leftrightarrow Y),\label{t2}
\eea
and
\bea
&&-\iota_{K^2[X,Y]}\frac{\omega^3}{3!}=\iota_{K[X,Y]}\varphi\wedge\varphi=-\frac{1}{2}\iota_{[X,Y]}F(\varphi)\wedge\varphi\nonumber\\
&=&-\frac{1}{2}(\cll_X\iota_YF(\varphi)-\iota_Y\cll_XF(\varphi))\wedge\varphi\nonumber\\
&=&-\frac{1}{2}(\ud\iota_X\iota_Y+\iota_X\ud\iota_Y-\iota_Y\ud\iota_X)F(\varphi)\wedge\varphi+\frac{1}{2}\iota_Y\iota_X\ud F(\varphi)\wedge\varphi.\label{t3}
\eea
Adding (\ref{t1}), (\ref{t2}) and (\ref{t3}) together, we get:
\bea
&&\iota_{N_K(X,Y)}\frac{\omega^3}{3!}\nonumber\\
&=&-\frac{1}{2}\ud\iota_{KY}\iota_{KX}\frac{\omega^3}{3!}+\frac{1}{2}\iota_X\varphi\wedge\ud\iota_YF(\varphi) -\frac{1}{4}\ud\iota_X\iota_YF(\varphi)\wedge\varphi\label{int}\\
&&+\iota_Y\varphi\wedge\iota_{KX}\ud\varphi+\iota_Y\iota_{KX}\ud\varphi\wedge\varphi+\frac{1}{4}\iota_Y\iota_X\ud F(\varphi)\wedge\varphi-(X\leftrightarrow Y).\nonumber
\eea
Notice that
\bea
&&\iota_{KY}\iota_{KX}\frac{\omega^3}{3!}=-\iota_{KY}(\iota_X\varphi\wedge\varphi)\nonumber\\
&=&\iota_X\iota_{KY}\varphi\wedge\varphi-\iota_X\wedge\iota_{KY}\varphi\nonumber\\
&=&-\frac{1}{2}\iota_X\iota_YF(\varphi)\wedge\varphi+\frac{1}{2}\iota_X\varphi\wedge\iota_YF(\varphi),\nonumber
\eea
Consequently we have
\bea
-\frac{1}{2}\ud\iota_{KY}\iota_{KX}\frac{\omega^3}{3!}&=&\frac{1}{4}\ud\iota_X\iota_YF(\varphi)\wedge\varphi +\frac{1}{4}\iota_X\iota_YF(\varphi)\wedge\ud\varphi\nonumber\\
&&-\frac{1}{4}\ud\iota_X\varphi\wedge\iota_YF(\varphi)-\frac{1}{4}\iota_X\varphi\wedge\ud\iota_YF(\varphi).\label{kk}
\eea
Plug (\ref{kk}) in (\ref{int}), we obtain
\bea
&&\iota_{N_K(X,Y)}\frac{\omega^3}{3!}\nonumber\\
&=&\frac{1}{4}\iota_X\varphi\wedge\ud\iota_YF(\varphi)-\frac{1}{4}\ud\iota_X\varphi\wedge\iota_YF(\varphi)+\frac{1}{4}\iota_X\iota_YF(\varphi)\wedge\ud\varphi \nonumber\\
&&+\iota_Y\varphi\wedge\iota_{KX}\ud\varphi+\iota_Y\iota_{KX}\ud\varphi\wedge\varphi+\frac{1}{4}\iota_Y\iota_X\ud F(\varphi)\wedge\varphi-(X\leftrightarrow Y)\nonumber\\
&=&\frac{1}{4}\iota_X\varphi\wedge(\cll_YF(\varphi)-\iota_Y\ud F(\varphi))-\frac{1}{4}(\cll_X\varphi-\iota_X\ud\varphi)\wedge\iota_YF(\varphi)+\frac{1}{4}\iota_X\iota_YF(\varphi)\wedge\ud\varphi\nonumber\\
&&+\frac{1}{4}\iota_Y\iota_X\ud F(\varphi)\wedge\varphi+\iota_Y\varphi\wedge\iota_{KX}\ud\varphi+\iota_Y\iota_{KX}\ud\varphi\wedge\varphi-(X\leftrightarrow Y)\label{crucial}.
\eea
Observe that
\bea
&&\iota_X\varphi\wedge\cll_YF(\varphi)-\cll_X\varphi\wedge\iota_YF(\varphi)-(X\leftrightarrow Y)\nonumber\\
&=&\iota_X\varphi\wedge\cll_YF(\varphi)+\cll_Y\varphi\wedge\iota_XF(\varphi)-(X\leftrightarrow Y)\nonumber\\
&=&\cll_Y(\iota_X\varphi\wedge F(\varphi)+\varphi\wedge\iota_XF(\varphi))-\cll_Y\iota_X\varphi\wedge F(\varphi)-\varphi\wedge\cll_Y\iota_XF(\varphi)-(X\leftrightarrow Y)\nonumber\\
&=&(\cll_X\iota_Y-\cll_Y\iota_X)\varphi\wedge F(\varphi)+\varphi\wedge(\cll_X\iota_Y-\cll_Y\iota_X)F(\varphi),
\eea
where we used that $\iota_X\varphi\wedge F(\varphi)+\varphi\wedge\iota_XF(\varphi)=0$ from (\ref{o1}). Since
\[\cll_X\iota_Y-\cll_Y\iota_X=\iota_{[X,Y]}+\iota_Y\cll_X-\cll_Y\iota_X=\iota_{[X,Y]}+\iota_Y\iota_X\ud-\ud\iota_Y\iota_X,\]
we get
\bea
&&\iota_X\varphi\wedge\cll_YF(\varphi)-\cll_X\varphi\wedge\iota_YF(\varphi)-(X\leftrightarrow Y)\nonumber\\
&=&\iota_Y\iota_X\ud\varphi\wedge F(\varphi)+\varphi\wedge\iota_Y\iota_X\ud F(\varphi)-\ud\iota_Y\iota_X\varphi\wedge F(\varphi)-\varphi\wedge\ud\iota_Y\iota_X F(\varphi).\nonumber
\eea
By taking a $\ud$ of (\ref{o22}), we get
\[\ud\iota_Y\iota_X\varphi\wedge F(\varphi)+\varphi\wedge\ud\iota_Y\iota_X F(\varphi)=\iota_Y\iota_X\varphi\wedge\ud F(\varphi)+\ud\varphi\wedge\iota_Y\iota_XF(\varphi),\]
hence
\bea
&&\iota_X\varphi\wedge\cll_YF(\varphi)-\cll_X\varphi\wedge\iota_YF(\varphi)-(X\leftrightarrow Y)\nonumber\\
&=&\iota_Y\iota_X\ud\varphi\wedge F(\varphi)+\varphi\wedge\iota_Y\iota_X\ud F(\varphi)-\iota_Y\iota_X\varphi\wedge\ud F(\varphi)-\ud\varphi\wedge\iota_Y\iota_XF(\varphi).\nonumber
\eea
Substitute it back into (\ref{crucial}), we get
\bea
&&\iota_{N_K(X,Y)}\frac{\omega^3}{3!}\nonumber\\
&=&\frac{1}{8}(\iota_Y\iota_X\ud\varphi\wedge F(\varphi)+\varphi\wedge\iota_Y\iota_X\ud F(\varphi)-\iota_Y\iota_X\varphi\wedge\ud F(\varphi)-\ud\varphi\wedge\iota_Y\iota_XF(\varphi))\nonumber\\
&&-\frac{1}{4}\iota_X\varphi\wedge\iota_Y\ud F(\varphi)+\frac{1}{4}\iota_X\ud\varphi\wedge\iota_YF(\varphi)-\frac{1}{4}\ud\varphi\wedge\iota_Y\iota_XF(\varphi)+\frac{1}{4}\varphi\wedge\iota_Y\iota_X\ud F(\varphi)\nonumber\\
&&+\iota_Y\varphi\wedge\iota_{KX}\ud\varphi+\iota_Y\iota_{KX}\ud\varphi\wedge\varphi-(X\leftrightarrow Y).\label{tosimp}
\eea
By taking double interior product of the equations $\ud\varphi\wedge F(\varphi)=0$, $\varphi\wedge\ud F(\varphi)=0$ and $\varphi\wedge\ud\varphi=0$, we get
\bea
\iota_X\ud\varphi\wedge\iota_YF(\varphi)-(X\leftrightarrow Y)&=&\iota_Y\iota_X\ud\varphi\wedge F(\varphi)+\ud\varphi\wedge\iota_Y\iota_XF(\varphi),\nonumber\\
-\iota_X\varphi\wedge\iota_Y\ud F(\varphi)-(X\leftrightarrow Y)&=&\iota_Y\iota_X\varphi\wedge\ud F(\varphi)+\varphi\wedge\iota_Y\iota_X\ud F(\varphi),\nonumber\\
\iota_Y\varphi\wedge\iota_{KX}\ud\varphi-\iota_Y\iota_{KX}\ud\varphi\wedge\varphi &=&-\frac{1}{2}\iota_Y\ud\varphi\wedge\iota_XF(\varphi)-\frac{1}{2}\ud\varphi\wedge\iota_Y\iota_XF(\varphi).\nonumber
\eea
From these identities we can conclude that
\bea
&&\iota_{N_K(X,Y)}\frac{\omega^3}{3!}\nonumber\\
&=&\varphi\wedge\iota_Y\iota_X\ud F(\varphi)+\iota_Y\iota_X\ud\varphi\wedge F(\varphi)-\ud\varphi\wedge\iota_Y\iota_X F(\varphi)+2\varphi\wedge(\iota_Y\iota_{KX}-\iota_X\iota_{KY})\ud\varphi.\nonumber
\eea
\end{proof}
As a corollary, we have proved
\begin{cor}
If $\varphi$ is integrable and $F$-integrable, then it is $K$-integrable.
\end{cor}
\begin{rmk}\label{rmk}
Notice that $K$-integrability implies that
\[[KX,KY]=K([KX,Y]+[X,KY]-K[X,Y])\in\img~K.\]
It follows from the Frobenius theorem that the distribution $\img~K(\varphi)$ is integrable. However, it seems that $K$-integrability is a stronger condition than that $\img~K(\varphi)$ is integrable as a distribution. For example, if $\varphi$ belongs to the orbit $\clo_0^\pm$ pointwise, then $\img~K(\varphi)=\ker F(\varphi)$, hence the integrability of $\img~K(\varphi)$ is guaranteed solely by the $F$-integrability of $\varphi$. 
\end{rmk} 

Next, we also prove that
\begin{thm}
If $\varphi$ is integrable and $F$-integrable, then it is $Q$-integrable.
\end{thm}
\begin{proof}
We need to prove that $Q(\varphi)$ is a constant on $M$ under the condition that $\ud\varphi=\ud F(\varphi)=0$. We can decompose $M$ as
\[M=\{Q<0\}\coprod\{Q=0\}\coprod\{Q>0\}.\]
As $\{Q<0\}$ and $\{Q>0\}$ are both open subsets of $M$, we only need to show that in every connected component of the set $\{Q\neq 0\}$, $Q$ is a constant. Let us assume that $U$ is a connected open subset of $\{Q<0\}$. In this case, we know that $\varphi$ defines an almost complex structure $J$ compatible with $\omega$ and that $F(\varphi)=|\varphi|^2\hat\varphi$ on $U$. Since $\ud\varphi=0$, in the language of \cite{fei2021b}, we get Type IIA structure $(\omega,\varphi)$ on $M$. In \cite[pp. 791-792]{fei2021b} we have shown that $\ud\hat\varphi$ is a $(2,2)$-form with respect to $J$. Consequently, we know that
\[\ud F(\varphi)=\ud|\varphi|^2\wedge\hat\varphi+|\varphi|^2\ud\hat\varphi\]
is exactly the decomposition of $\ud F(\varphi)$ into its $(3,1)+(1,3)$ and $(2,2)$ components. Therefore $\ud F(\varphi)=0$ implies that both $\ud|\varphi|^2=0$ and $\ud\hat\varphi=0$, which further implies that $Q=-|\varphi|^4$ is a constant and $J$ is integrable.

The situation that $U$ is contained in $\{Q>0\}$ is parallel, see the analogous construction outlined in \cite{fei2015b}.
\end{proof}
To summarize, we have proved that
\begin{thm}\label{integ}
If $\varphi$ is integrable and $F$-integrable, then it is also $K$-integrable and $Q$-integrable. Moreover, $\varphi$ is either pointwise stable or pointwise unstable. It is impossible for $\varphi$ to be stable at some points while unstable at other places. 
\end{thm}
\begin{rmk}
There are examples of $\varphi$ being both integrable and $F$-integrable with $Q(\varphi)=0$ everywhere such that $\varphi$ belongs to different non-stable orbits at different points.
\end{rmk}

\begin{rmk}
Being both integrable and $F$-integrable is a very natural condition to impose. For example, when $\varphi$ belongs to the orbit $\clo_-^+(\mu)$ pointwise, being both integrable and $F$-integrable is equivalent to that the associated metric $q(\omega,\varphi)$ is both K\"ahler and Ricci-flat. If the underlying manifold is a K\"ahler Calabi-Yau 3-fold, such a $\varphi$ can always be found by running the Type IIA flow with the real part of any holomorphic volume form as initial data, see \cite[Theorem 9]{fei2021b}. This condition also fits well into what Hitchin calls the ``nonlinear Hodge theory''. In fact, on a closed symplectic 6-manifold $(M,\omega)$, one can introduce the a functional $I$ on the space of closed 3-forms within a fixed cohomology class:
\[I(\varphi)=\int_MQ(\varphi)\frac{\omega^3}{3!}.\]
One can show that the critical points of $I$ are exactly those 3-forms that are both integrable and $F$-integrable. Therefore it is naturally to introduce the following definition.
\begin{dfn}
We say a primitive 3-form on $(M,\omega)$ is $F$-harmonic if it is both integrable and $F$-integrable.
\end{dfn}
\end{rmk}
 
\section{Lagrangian foliations and degeneration of Calabi-Yau structures}

In this section, we examine how certain Lagrangian foliations can be regarded as degeneration of Calabi-Yau structures and discuss its potential interactions with the Stroinger-Yau-Zaslow (SYZ) conjecture \cite{strominger1996}.

Roughly speaking, the SYZ conjecture predicts that near the large complex structure limit, which is a place where the complex structure of a Calabi-Yau manifold has maximal degeneration, the Calabi-Yau manifold admits a fibration of special Lagrangian tori with singularities. Moreover, the base of the fibration carries a pair of dual Hessian structures (with singularities!) such that the local potentials solve a real Monge-Amp\`ere equation. The SYZ conjecture has served as a guiding principle to understand the mirror symmetry from a geometric point of view. There has been numerous work devoted to various aspects of the SYZ conjecture. We recommend Yang Li's survey \cite{li2022b} and references therein for its recent developments.

The usual approach to the SYZ conjecture factors through the semi-flat geometry, which serves as an asymptotic model of the Ricci-flat K\"ahler metrics near the large complex structure limit. In this section, we shall consider a weaker version of the SYZ picture, namely how a Lagrangian foliation naturally arises from certain degeneration of complex structures on a 6-dimensional Calabi-Yau manifold, where semi-flatness is an automatic consequence of the special structures of ambient geometry.

Let $(M,\omega)$ be a fixed compact symplectic 6-manifold with vanishing first Chern class. Let $J_t$ be a family of integrable complex structures on $M$ that are compatible with $\omega$. Moreover, we require that $(M,J_t)$ has holomorphically trivial canonical bundle so we get a family of Calabi-Yau structures. Such a family can be obtained by applying the Moser's trick to a usual polarized degeneration of Calabi-Yau manifolds. Let $\Omega_t=\varphi_t+\sqrt{-1}\hat\varphi_t$ be a family of holomorphic volume forms on $(X,J_t)$. From Section 2, we know that all the information encoded in the degeneration of complex structures can be recovered from the family $\varphi_t$, which is a family of real primitive 3-forms on $M$ belonging to the orbit $\clo_-^+(\mu)$ pointwise. In fact, we shall further impose that $\varphi_t$ is Ricci-flat, meaning that we have $\ud\hat\varphi_t=0$ and $|\varphi_t|^2$ is a constant on $M$. This can always be achieved by running the Type IIA flow \cite[Theorem 9]{fei2021b}. Consequently $F(\varphi_t)=|\varphi_t|^2\hat\varphi_t$ is also $\ud$-closed.

Suppose the limit $\lim_{t\to\infty}\varphi_t=\varphi_\infty$ exists smoothly on $M$, where $\varphi_\infty$ is a primitive 3-form belonging to the orbit $\clo_0^+$ or $\clo_3$ pointwise on an open dense subset $U$ of $M$. Since $\varphi_t$ is $F$-harmonic and the convergence is smooth, we see that $\varphi_\infty$ is also $F$-harmonic. In view of Proposition \ref{cases} and Remark \ref{rmk}, we know that there is canonically a Lagrangian foliation associated to $\varphi_\infty$ on $U$, which can be viewed as a Lagrangian foliation with singularities on $M$.

\begin{ex}
Let $T=\rr^2/\zz^2$ be the standard 2-torus equipped with the symplectic form $\omega=\ud x\wedge\ud y$. For any $\tau$ in the upper half plane, consider the complex structure $J_\tau$ on $T$ defined by the holomorphic $(1,0)$-form $\ud x+\tau\ud y$ on $T$. It is not hard to see that as a Riemann surface, $(T,J_\tau)$ is biholomorphic to $\ccc/(\zz+\zz\tau)$, the elliptic curve with period $\tau$. For this reason we shall write $E_\tau$ to denote the Calabi-Yau 1-fold $(T,\omega,J_\tau)$. 

Now let $M_\tau=E_\tau\times E_\tau\times E_\tau$ with product complex and symplectic structures. It is easy to see that $M_\tau$ is a Calabi-Yau 3-fold with K\"aher Ricci-flat metric. Moreover, it is well-known that the family $M_\tau$ with $\tau\to\sqrt{-1}\infty$ is a model of large complex structure limit. For example, it is indicated in \cite{leung2005} that under the limit $\tau\to\sqrt{-1}\infty$, the geometry of $M_\tau$ collapses to a 3-dimensional base.

Now let us consider the following holomorphic volume form $\Omega_\tau$ on $M_\tau$:
\[\Omega_\tau=\frac{1}{\tau^2}(\ud x^1+\tau\ud y^1)\wedge(\ud x^2+\tau\ud y^2)\wedge(\ud x^3+\tau\ud y^3),\]
with $\varphi_\tau=\textrm{Re}~\Omega_\tau$. In this set-up we have
\[\lim_{\tau\to\sqrt{-1}\infty}\varphi_\tau=\ud x^1\wedge\ud y^2\wedge\ud y^3+\ud x^2\wedge\ud y^3\wedge\ud y^1+\ud x^3\wedge\ud y^1\wedge\ud y^2=:\varphi_\infty\]
is a member of $\clo_3^+$. Associated to $\varphi_\infty$ there is a natural Lagrangian foliation 
\[L=\ker K_{\varphi_\infty}=\spann\left\{\frac{\p}{\p x^1},\frac{\p}{\p x^2},\frac{\p}{\p x^3}\right\}.\]
In fact, in this example $L$ is a Lagrangian torus fibration over a 3-torus base.
\end{ex}

The case that $\varphi_\infty$ belonging to the $\clo_0^+$ is the more desired situation. Because in this case, the other sequence $F(\varphi_t)$ satisfies
\[\lim_{t\to\infty} F(\varphi_t)=F(\varphi_\infty),\]
where $F(\varphi_\infty)$ belongs to the orbit $\clo_3$. 

For this reason, to understand the degeneration of Calabi-Yau structures, we shall focus on the geometry associated to the triple $(M,\omega,\varphi)$, where $(M,\omega)$ is a symplectic 6-manifold, and $\varphi$ is a primitive 3-form on $M$ belonging to the orbit $\clo_0^+$ pointwise that is also $F$-harmonic.

Summarizing results from Proposition \ref{cases} and Theorem \ref{integ}, we have the following basic description of some fundamental geometric structures associated to $(M,\omega,\varphi)$.
\begin{prop}
Let $(M,\omega)$ be a symplectic 6-manifold and $\varphi$ an $F$-harmonic primitive 3-form on $M$ belonging to the orbit $\clo_0^+$ pointwise. Then $K=K(\varphi)$ is a smooth section of $\endo~TM$ such that $K^2=0$ and $\cll:=\ker K=\img~K$ defines a Lagrange foliation on $(M,\omega)$. In addition, the Nijenhuis tensor $N_K$ of $K$ vanishes, namely
\bea
N_K(X,Y)=K[KX,Y]+K[X,KY]-[KX,KY]=0\label{vanishing}
\eea
for any vector fields $X$ and $Y$. The natural bilinear symmetric form $g(\cdot,\cdot)=\omega(\cdot,K\cdot)$ induces a Riemannian metric on the vector bundle $TM/\cll$. The bundle $TM/\cll$ is canonically oriented so the volume form of the above Riemannian metric is exactly given by $\dfrac{1}{\sqrt{2}}F(\varphi)$. Through the bundle isomorphism $K:TM/\cll\cong \cll$, we also get a natural Riemannian metric $g_\cll$ with orientation on the Lagrangian foliation $\cll$ that
\[g_\cll(X,Y)=\omega(K^{-1}X,Y)\]
for any vector fields tangent to $\cll$. Clearly the definition of $g_\cll$ is independent of the choice of preimage $K^{-1}X$ of $X$ since $\cll$ is Lagrangian. 
\end{prop} 
It turns out that there are very rich geometric structures on the leaves of the Lagrangian foliation $\cll$. In particular we have the following theorem.
\begin{thm}\label{structure}
Let $L$ be a leaf, namely a maximal integral (immersed) submanifold, of the Lagrangian foliation $\cll$. Then $L$ is naturally oriented. Moreover, $L$ carries a canonical affine structure such that the natural Riemannian metric induced from $g_\cll$ is a Hessian metric with parallel volume form. In other words, in any local affine coordinate system on $L$, the metric is the Hessian of a locally defined convex function $h$ such that it solves a real Monge-Amp\`ere equation with constant right hand side.
\end{thm}
\begin{proof}
Since $\cll=\ker F(\varphi)$, we know that $F(\varphi)$ naturally defines an orientation on the vector bundle $TM/\cll$. Therefore we also get a natural orientation on $\cll$, hence every leaf of $\cll$ is naturally oriented as well.

Now let us introduce a partial connection $D$ on $\cll$ as follows: for any vector fields $X$ and $Y$ that are tangent to $\cll$ everywhere, we define
\bea
D_XY=K[X,K^{-1}Y].\label{dualbott}
\eea
Suppose we have two choices $Z_1$ and $Z_2$ of $K^{-1}Y$, since $K(Z_1)=K(Z_2)=Y$, we know that $Z_1-Z_2$ is a section of $\ker K=\cll$. As $\cll$ is integrable and $X$ is tangent to $\cll$, so $[X,Z_1-Z_2]$ is also tangent to $\cll$ hence $K[X,Z_1-Z_2]=0$. Therefore $D$ is well defined as $D_XY$ does not depend on the choice of $K^{-1}Y$. One can check that $D$ satisfies the usual properties of a connection. For any function $f$, we have
\[D_{fX}Y=K[fX,K^{-1}Y]=K(f[X,K^{-1}Y])-(K^{-1}Yf)KX=fD_XY,\]
and
\[\begin{split}
D_X(fY)&=K[X,K^{-1}(fY)]=K[X,fK^{-1}Y]=K(f[X,K^{-1}Y])+Xf\cdot K(K^{-1}Y)\\
&=fD_XY+Xf\cdot Y.\end{split}\]
Moreover, by vanishing of the Nijenhuis tensor (\ref{vanishing}), we have
\[\begin{split}D_XY-D_YX=&K([K(K^{-1}X),K^{-1}Y]+[K^{-1}X,K(K^{-1}Y)])=[K(K^{-1}X),K(K^{-1}Y)]\\
=&[X,Y].\end{split}\]
The Jacobi identity further suggests that
\[\begin{split}&D_X(D_YZ)-D_Y(D_XZ)-D_{[X,Y]}Z\\
=&K[X,[Y,K^{-1}Z]]-K[Y,[X,K^{-1}Z]]-K[[X,Y],K^{-1}Z]\\
=&0.\end{split}\]
Consequently the restriction of $D$ to each leaf $L$ is a torsion-free and flat connection on $TL$, hence we get a canonical affine structure on every leaf.

Let $g_L$ be the induced Riemannian metric from $g_\cll$. Then for any tangent vectors $Y$ and $Z$ of $L$, we have $g_L(Y,Z)=\omega(K^{-1}Y,Z)$,
hence
\[D_Xg_L(Y,Z)=X\omega(K^{-1}Y,Z)-\omega([X,K^{-1}Y],Z)-\omega(K^{-1}Y,D_XZ).\]
On the other hand, $\omega$ is a symplectic form, so
\[\begin{split}0=\ud\omega(X,K^{-1}Y,Z)=&X\omega(K^{-1}Y,Z)+K^{-1}Y\omega(Z,X)+Z\omega(X,K^{-1}Y)\\
&-\omega([X,K^{-1}Y],Z)-\omega([K^{-1}Y,Z],X)-\omega([Z,X],K^{-1}Y),\end{split}\]
which can be translated to
\[D_Xg_L(Y,Z)=D_Zg_L(Y,X),\]
hence $Dg_L$ is a totally symmetric tensor. In a local affine coordinate system $\{x^1,x^2,x^3\}$, we can write $g_L=h_{ij}\ud x^i\otimes\ud x^j$. It follows that 
\[\frac{\p}{\p x^k}h_{ij}=\frac{\p}{\p x^j}h_{ik}\]
for any indices $i,j$ and $k$. Therefore locally we can find a potential function $h$ such that $h_{ij}=\dfrac{\p^2h}{\p x^i\p x^j}$ hence $g_L$ is a Hessian metric. Let $\vol_{g_L}$ be the volume form of the metric $g_L$. By a pointwise computation using the normal form, we have
\[\vol_{g_L}(X,Y,Z)=\frac{1}{\sqrt{2}}F(\varphi)(K^{-1}X,K^{-1}Y,K^{-1}Z).\]
Since $\ud F(\varphi)=0$ and $\cll=\ker F(\varphi)$, we get
\[\begin{split}&WF(\varphi)(K^{-1}X,K^{-1}Y,K^{-1}Z)-F(\varphi)([W,K^{-1}X],K^{-1}Y,K^{-1}Z)\\
&-F(\varphi)(K^{-1}X,[W,K^{-1}Y],K^{-1}Z)-F(\varphi)(K^{-1}X,K^{-1}Y,[W,K^{-1}Z])=0\end{split}\]
for any $X,Y,Z,W$ tangent to $\cll$. This equation translates to
\[D_W\vol_{g_L}(X,Y,Z)=0.\]
Therefore we know that $\vol_{g_L}$ is parallel under $D$.
\end{proof}

\begin{rmk}\label{foliationcoor}
It is not hard to write down a locally $D$-parallel frame of $TL$. By a theorem of Weinstein \cite{weinstein1971}, near any point $p$ of $M$, there exists locally a Darboux coordinate system $\{x^1,y^1,x^2,y^2,x^3,y^3\}$ such that
\[\omega=\ud x^1\wedge\ud y^1+\ud x^2\wedge\ud y^2+\ud x^3\wedge\ud y^3,\]
and the leaves of $\cll$ are locally defined by $y=(y^1,y^2,y^3)$ are constants. In this local coordinates, $\left\{K\dfrac{\p }{\p y^j}\right\}_{j=1}^3$ gives rise to a $D$-parallel frame on $TL$.
\end{rmk}

On each leaf $L$, the pair $(g_L,D)$ forms what is known as a Hessian structure \cite[Chap. 2]{shima2007}. By performing the Legendre transform locally, one can generate a dual Hessian structure $(g_L,D')$ on $L$, where $D'$ is a torsion-free and flat connection on $TL$ such that
\[D+D'=2\nabla,\]
where $\nabla$ is the Levi-Civita connection associated to $g_L$. The same metric $g_L$ is also Hessian with respect to $D'$ and the volume form $\vol_{g_L}$ is also parallel under $D'$. On the other hand, for any Lagrangian foliation, there is another canonical affine structure on its leaves characterized by the Bott (partial) connection $\nabla^B$ \cite{bott1972}. In this setting, see for example \cite[Chap. 4]{hamilton2023}, the Bott (partial) connection satisfies
\bea
\omega(\nabla^B_XY,Z)=X\omega(Y,Z)+\omega([X,Z],Y)\label{bott}
\eea
for any $X,Y$ tangent to $\cll$ and arbitrary vector field $Z$. It is not surprising to expect the following theorem:
\begin{thm}
The Bott connection $\nabla^B$ is dual to $D$ in the sense of Hessian structures, namely $\nabla^B=D'$. In particular, with respect to the affine structure determined by $\nabla^B$, the metric $g_L$ on each leaf $L$ is Hessian and its local potential functions all solve a real Monge-Amp\`ere equation with constant right hand side.
\end{thm}
\begin{proof}
We only need to prove that
\[Xg_L(Y,Z)=g_L(D_XY,Z)+g_L(Y,\nabla^B_XZ)\]
for any vector fields $X,Y,Z$ tangent to $\cll$. This follows directly from that $g_L(\cdot,\cdot)=\omega(K^{-1}\cdot,\cdot)$ and the definitions of $D$ (\ref{dualbott}) and $\nabla^B$ (\ref{bott}).
\end{proof}
\begin{thm}
The natural metric $g_L$ on any leaf $L$ has nonnegative Ricci curvature. If the leaf $L$ is compact, then it must be a flat torus.
\end{thm}
\begin{proof}
In an affine coordinate, the Ricci curvature $R_{jk}$ of a Hessian metric $g$ is given by
\bea
R_{jk}=\frac{1}{4}h^{st}h^{lp}(h_{jps}h_{klt}-h_{jkt}h_{pls}),\label{riccigen}
\eea
where $h$ is the local potential function for the Hessian metric $g$ and $h_{ijk}$ is the short-handed notation for $\dfrac{\p^3h}{\p x^i\p x^j\p x^k}$, the 3rd derivative of $h$ with respect to an affine coordinate system. If in addition that the volume form is parallel under the affine connection, then $\det[h_{pl}]$ is a constant, hence
\[\frac{\p}{\p x^s}\log\det[h_{pl}]=h^{lp}h_{pls}=0.\]
Plugging it in (\ref{riccigen}), we get
\bea
R_{jk}=\frac{1}{4}h^{st}h^{lp}h_{jps}h_{klt}
\eea
is semi-positive definite. In other words, for any vector field $X$, we have
\bea
\textrm{Ric}(X,X)=\frac{1}{4}|D_Xg|^2.\label{ricci}
\eea
It follows that the scalar curvature $S$ has the expression
\bea
S=\frac{1}{4}h^{st}h^{ik}h^{jl}h_{sil}h_{tkj}=\frac{1}{4}|Dg|^2\geq 0.\label{scalar}
\eea
As in Theorem (\ref{structure}) we have proved that for any leaf $L$, it has a Hessian structure with parallel volume form, so we know its Ricci curvature is nonegative from (\ref{ricci}).

Now suppose that $L$ is a compact leaf. It is well-known that $L$ must be a torus. By a famous theorem of Schoen-Yau \cite{schoen1979} and Gromov-Lawson \cite{gromov1980}, we know that any Riemannian metric on a torus with nonnegative scalar curvature must be flat.
\end{proof}
\begin{rmk}
In general, the two affine structures determined by $D$ and $\nabla^B$ are different. This follows from the fact that the metric $g_L$ on each leaf is in general not flat. See the example below.
\end{rmk}
\begin{ex}
Let $(N,g)$ be an oriented 3-dimensional Riemannian manifold. Over $N$, we consider its cotangent bundle $T^*N$, namely the bundle of 1-forms, and also the bundle of 2-forms $\bigwedge^2T^*N$. It is well-known that $T^*N$ carries a tautological 1-form, whose exterior derivative gives rise to the canonical symplectic structure $\omega$ on $T^*N$. Let $\{x^1,x^2,x^3\}$ be a local coordinate chart on $N$, and let $\{y_1,y_2,y_3\}$ be the corresponding coordinate in the cotangent direction, we know that $\omega$ takes the Darboux form
\[\omega=\ud x^1\wedge\ud y_1+\ud x^2\wedge\ud y_2+\ud x^3\wedge\ud y_3.\]
Similarly, the total space of $\bigwedge^2T^*N$ carries a tautological 2-form $\alpha$ defined as follows: Let $(p,\lambda)$ be a point in $\bigwedge^2T^*N$ where $p\in N$ and $\lambda\in\bigwedge^2T^*_pN$. For any $X,Y\in T_{(p,\lambda)}\bigwedge^2T^*N$, we define that
\[\alpha_{(p,\lambda)}(X,Y)=\lambda(\pi_*X,\pi_*Y),\]
where $\pi:\bigwedge^2T^*N\to N$ is the projection map. Over a local coordinate chart $\{x^1,x^2,x^3\}$ of $N$, any 2-form on $N$ can be written as $t^1\ud x^2\wedge\ud x^3+t^2\ud x^3\wedge\ud x^1+t^3\ud x^1\wedge\ud x^2$. If we use $\{x^1,t^1,x^2,t^2,x^3,t^3\}$ as a local coordinate chart on $\bigwedge^2T^*N$, we can express $\alpha$ as
\[\alpha=t^1\ud x^2\wedge\ud x^3+t^2\ud x^3\wedge\ud x^1+t^3\ud x^1\wedge\ud x^2.\]
By taking the exterior derivative, we get the canonical 3-form $\ud\alpha$ on $\bigwedge^2T^*N$:
\[\begin{split}\ud\alpha&=\ud(t^1\ud x^2\wedge\ud x^3+t^2\ud x^3\wedge\ud x^1+t^3\ud x^1\wedge\ud x^2)\\
&=\ud t^1\wedge\ud x^2\wedge\ud x^3+\ud t^2\wedge\ud x^3\wedge\ud x^1+\ud t^3\wedge\ud x^1\wedge\ud x^2.\end{split}\]

The Riemannian metric $g$ on $N$ provides us an isomorphism $T^*N\cong\bigwedge^2T^*N$ via the Hodge star operator. In the local coordinates above, we have
\[t^j=y_kg^{kj}\sqrt{\det g},\]
or
\[y_k=\frac{1}{\sqrt{\det g}}g_{kj}t^j.\]
Under this identification, we can think of $\omega$ as a symplectic form on $\bigwedge^2T^*N$ with the expression
\[\omega=\frac{g_{kj}}{\sqrt{\det g}}\ud x^k\wedge\ud t^j+t^j\ud x^k\wedge\ud\left(\frac{g_{kj}}{\sqrt{\det g}}\right).\]
It is not hard to find out that
\[\frac{\omega^3}{3!}=\frac{1}{\sqrt{\det g}}\ud x^1\wedge\ud t^1\wedge\ud x^2\wedge\ud t^2\wedge\ud x^3\wedge\ud t^3.\]
Let $r$ be the norm square function on $\bigwedge^2T^*N$, namely $r(p,\lambda)=|\lambda|^2_{g_p}$. In local coordinates we have that
\[r(x,t)=\frac{t^jt^k}{\det g}g_{jk}.\]
A straightforward calculation shows that 
\bea
\alpha\wedge\omega=\frac{\sqrt{\det g}}{2}\ud x^1\wedge\ud x^2\wedge\ud x^3\wedge\ud r.\label{basic}
\eea
It follows that $\ud\alpha\wedge\omega=\ud(\alpha\wedge\omega)=0$, hence $\ud\alpha$ is primitive.

Let $f$ be a smooth single variable function of $r$, and we define
\[\varphi_f=\ud(f\alpha)=f\ud\alpha+f'\ud r\wedge\alpha.\]
It follows from (\ref{basic}) that $\varphi_f$ is a primitive 3-form for any such $f$. Moreover, direct calculation reveals that
\[\begin{split}&K(\varphi_f)\left(\frac{\p}{\p t^j}\right)=0,\\
&K(\varphi_f)\left(\frac{\p}{\p x^j}\right)=2f\sqrt{\det g}\left[(f+2rf')\frac{\p}{\p t^j}-f'\frac{\p r}{\p t^j}\cle\right],\end{split}\]
where $\cle=t^j\dfrac{\p}{\p t^j}$ is the local Euler vector field. It follows that 
\[F(\varphi_f)=-4f^2(f+2rf')\sqrt{\det g}\cdot\ud x^1\wedge\ud x^2\wedge\ud x^3.\]
We are mainly interested in the case that 
\bea
f(r)=r^{-\frac{1}{2}}(r^{\frac{3}{2}}+C)^{\frac{1}{3}}\label{fchoice}
\eea
for some constant $C$, so that
\[f^2(f+2rf')=1.\]
With the above choice (\ref{fchoice}) of $f$, we consider $M=\bigwedge^2T^*N$ when $C=0$ or $M=\bigwedge^2T^*N\setminus N$ when $C>0$ and $M=\bigwedge^2T^*N
\setminus\{r\leq(-C)^{2/3}\}$ when $C<0$. It follows that $(M,\omega,\varphi_f)$ is a triple such that $(M,\omega)$ is a symplectic 6-manifold and $\varphi_f$ is a primitive 3-form on $M$ belonging to the orbit $\clo_0^+$ pointwise. In addition, $\varphi_f$ is $F$-harmonic.

It is obvious that the Lagrangian foliation $\cll$ in this case is simply given by the vertical foliation of the fibration $\pi:\bigwedge^2T^*N\to N$. On one hand, the affine structure associated to the Bott connection $\nabla^B$ is the standard affine structure on the fibers of the vector bundle $\pi:\bigwedge^2T^*N\to N$. To understand the affine structure associated to $D$, we let 
\[V_j=K(\varphi_f)\left(\frac{\p}{\p x^j}\right)=2f\sqrt{\det g}\left[(f+2rf')\frac{\p}{\p t^j}-f'\frac{\p r}{\p t^j}\cle\right],\]
for $j=1,2,3$. By Remark \ref{foliationcoor} we know that the restriction of $V_j$ to any leaf is parallel under $D$, hence gives rise to a local affine coordinate on the Lagrangian leaf. In this local affine coordinate, the induced metric $g_{\cll}$ takes the Hessian form
\[h_{jk}=2f^{-1}g_{jk}-ff'\det g\frac{\p r}{\p t^j}\frac{\p r}{\p t^k}.\]
By direct computation we have
\[\det [h_{jk}]=8\det g\]
is constant in $t$, as predicted. In addition, we have
\[h^{jk}=\frac{f}{2}\left[g^{jk}+\frac{f'\det g}{2f}g^{jp}\frac{\p r}{\p t^p}\frac{\p r}{\p t^q}g^{qk}\right],\]
and
\[h_{jkl}=-4f'\sqrt{\det g}\left(g_{jk}\frac{\p r}{\p t^l}+g_{kl}\frac{\p r}{\p t^j}+g_{lj}\frac{\p r}{\p t^k}\right)-\frac{C(5r^{\frac{3}{2}}+2C)}{2r^4(r^{\frac{3}{2}}+C)^{\frac{2}{3}}}(\det g)^{\frac{3}{2}}\frac{\p r}{\p t^j}\frac{\p r}{\p t^k}\frac{\p r}{\p t^l}.\]
To demonstrate that $D$ and $\nabla^B$ induces different affine structures on the leaves, we only need to show that the induced metric $g_L$ on the leaves are not flat, as $D$ and $\nabla^B$ define dual Hessian structures on every leaf. This can be done by computing the scalar curvature of every leaf $L$. By (\ref{scalar}) we know that
\[S=\frac{1}{4}|Dg|^2>0\]
because $h_{jkl}\neq 0$.

Alternatively, we notice that in this example the induced metric on the leaves are rotationally symmetric. If we let $r=\rho^2$, then this metric can be written as
\[g_L=\frac{\rho^2}{2(\rho^3+C)^{2/3}}(\ud\rho)^2+\frac{\rho(\rho^3+C)^{1/3}}{2}g_{S^2}.\]
It follows that the scalar curvature of $g_L$ is
\[S=\frac{5C^2}{\rho^4(\rho^3+C)^{4/3}}.\]
\end{ex}

\begin{rmk}
Much of the discussion in this section can be placed in a more general setting such as in higher dimensions or with weaker assumption on integrability conditions. We shall not pursue this kind of generality in this paper. The general theory will be developed elsewhere.
\end{rmk}

\section{Detecting geometrical structures on symplectic 6-manifolds}

In a series of papers \cite{fei2021b},\cite{fei2023},\cite{fei2021g}, we introduce and investigate the Type IIA flow as an intrinsically defined geometric flow on symplectic 6-manifolds. In the language of Section 2, the Type IIA flow on a fixed symplectic manifold $(M,\omega)$ is the following evolution equation of 3-forms
\bea
\pt_t\varphi=\ud\Lambda_\omega\ud F(\varphi),\label{iiaflow}
\eea
where $\Lambda_\omega$ is the Lefschetz operator of contraction with respect to the symplectic form $\omega$. Usually we require the initial data $\varphi\big|_{t=0}=\varphi_0$ to be a $\ud$-closed 3-form that belongs to the orbit $\clo_-^+(\mu)$ pointwise so we have the short-time existence and uniqueness. However, in principle Equation (\ref{iiaflow}) makes sense for any initial data though additional effort is in need for its existence.

In this section, We also use the Type IIA flow as a tool to detect geometric structures on certain symplectic 6-manifolds, notably those with large symmetry, from two perspectives which are complementary to each other. 

On one hand, we would like to find integrable and closed stationary points of the Type IIA flow, namely primitive 3-forms $\varphi$ satisfying
\[\begin{cases}&\ud\varphi=0,\\
&\ud\Lambda_\omega\ud F(\varphi)=0.\end{cases}\]
In particular, these stationary points include those $\varphi$ that are $F$-harmonic. Presumably, the set of all $F$-harmonic 3-forms encode rich geometric and topological information about the underlying symplectic manifold. Moreover, the structure of this set also tells us what kind of geometric structures can be supported on this manifold. For example, by the result in \cite{fei2021g}, we know such $\varphi$ cannot pointwise lie in the orbit $\clo_-^+(\mu)$ if the underlying manifold does not support any K\"ahler structure.  

On the other hand, we would like to study the limits of the Type IIA flow. Suppose we have a Type IIA flow $\varphi(t)$ with initial data from $\clo_-^+(\mu)$ on a maximal time interval $[0,T)$. In general, we do not expect its convergence when $t\to T$ as the underlying manifold $(M,\omega)$ may be non-K\"ahler. However, it is possible that after a choice of suitable normalizations, the sequence $C(t)\varphi(t)$ has a smooth limit $\varphi_\infty$ when $t\to T$, where $C(t)$ is a constant depending on $t$. In this case, natural geometric structures associated to $\varphi_\infty$ would be of great interest, as they are special in the sense that they can be detected by the intrinsically defined Type IIA flow.

In this section, we shall test these ideas on certain 6-dimensional Lie groups with left-invariant symplectic structures and their quotients by cocompact discrete subgroups. In one aspect, these manifolds are rich enough to demonstrate various phenomena in Type IIA flow \cite[Section 9.3.2]{fei2021b}. On the other hand, their large group of local symmetries allows us to reduce the Type IIA flow to an ODE system that sometimes can be solved explicitly. Moreover, as we have the existence theorems about solutions to ODE, we could even include degenerate and non-closed 3-forms as out initial data.

For simplicity, we shall consider only those 6-dimensional symplectic Lie groups with cocompact discrete subgroups. It is well-known (see for example \cite{milnor1976}) that any Lie group admitting a cocompact discrete subgroup must be unimodular. By a theorem of Chu \cite{chu1974}, any unimodular symplectic Lie group must be solvable. Though the classification is not fully complete, the structure of 6-dimensional symplectic solvable Lie groups is well-understood from the works \cite{goze1987, goze1996, salamon2001, khakimdjanov2004, campoamor-stursberg2009, baues2016}. In particular, it contains many more examples compared to 6-dimensional Lie groups with left-invariant symplectic half-flat structures \cite{conti2007, fernandez2013}. In this section, we shall not attempt to exhaust all the possible cases. Instead, we present a more refined analysis of the two examples given in \cite[Section 9.3.2]{fei2021b}. Further examples will be left for future studies.

It is convenient to have the following lemma for computational purpose.
\begin{lemma}\label{bc}
Choose a basis $\{e^1,e^2,\dots,e^6\}$ of $V^*$ such that the symplectic form $\omega$ takes the standard form
\[\omega=e^{12}+e^{34}+e^{56}.\]
A general primitive 3-form $\varphi$ on $V$ can be written as
\bea
\varphi&=&Ae^{135}+Be^{136}+Ce^{145}+De^{146}+Ee^{235}+Fe^{236}+Ge^{245}+He^{246}\label{generalphi}\\
&&+(Ie^1+Je^2)(e^{34}-e^{56})+(Ke^3+Le^4)(e^{12}-e^{56})+(Me^5+Ne^6)(e^{12}-e^{34}),\nonumber
\eea
where $A,B,\dots,N$ are some constants. Then $F(\varphi)$ satisfies
\bea
-\frac{1}{2}F(\varphi)&=&\hat Ae^{135}+\hat Be^{136}+\hat Ce^{145}+\hat De^{146}+\hat Ee^{235}+\hat Fe^{236}+\hat Ge^{245}+\hat He^{246}\label{generalf}\\
&&+(\hat Ie^1+\hat Je^2)(e^{34}-e^{56})+(\hat Ke^3+\hat Le^4)(e^{12}-e^{56})+(\hat Me^5+\hat Ne^6)(e^{12}-e^{34}),\nonumber
\eea
where
\bea
\hat A&=&A(AH-BG-CF-DE+2IJ+2KL+2MN)\nonumber\\
&&-2(BM^2+CK^2+EI^2-BCE+2IKM),\nonumber\\
\hat B&=&B(AH-BG+CF+DE+2IJ+2KL-2MN)\nonumber\\
&&-2(-AN^2+DK^2+FI^2+ADF+2IKN),\nonumber\\
\hat C&=&C(AH+BG-CF+DE+2IJ-2KL+2MN)\nonumber\\
&&-2(-AL^2+DM^2+GI^2+ADG+2ILM),\nonumber\\
\hat D&=&D(-AH-BG-CF+DE+2IJ-2KL-2MN)\nonumber\\
&&-2(-BL^2-CN^2+HI^2-BCH+2ILN),\nonumber
\eea
and
\bea
\hat E&=&E(AH+BG+CF-DE-2IJ+2KL+2MN)\nonumber\\
&&-2(-AJ^2+FM^2+GK^2+AFG+2JKM),\nonumber\\
\hat F&=&F(-AH-BG+CF-DE-2IJ+2KL-2MN)\nonumber\\
&&-2(-BJ^2+HK^2-EN^2-BEH+2JKN),\nonumber\\
\hat G&=&G(-AH+BG-CF-DE-2IJ-2KL+2MN)\nonumber\\
&&-2(-CJ^2-EL^2+HM^2-CEH+2JLM),\nonumber\\
\hat H&=&H(-AH+BG+CF+DE-2IJ-2KL-2MN)\nonumber\\
&&-2(-DJ^2-FL^2-GN^2+DFG+2JLN),\nonumber\\
\hat I&=&I(AH-BG-CF+DE)-2J(AD-BC)+2(ALN-BLM-CKN+DKM),\nonumber\\
\hat J&=&J(-AH+BG+CF-DE)+2I(EH-FG)+2(ELN-FLM-GKN+HKM),\nonumber\\
\hat K&=&K(AH-BG+CF-DE)-2L(AF-BE)+2(AJN-BJM-EIN+FIM),\nonumber\\
\hat L&=&L(-AH+BG-CF+DE)+2K(CH-DG)+2(CJN-DJM-GIN+HIM),\nonumber\\
\hat M&=&M(AH+BG-CF-DE)-2N(AG-CE)+2(AJL-CJK-EIL+GIK),\nonumber\\
\hat N&=&N(-AH-BG+CF+DE)+2M(BH-DF)+2(BJL-DJK-FIL+HIK).\nonumber
\eea
In this frame, we have
\bea
&&\frac{Q(\varphi)}{4}\label{generalq}\\
&=&2(A^2H^2+B^2G^2+C^2F^2+D^2E^2)-(AH+BG+CF+DE)^2+4(ADFG+BCEH)\nonumber\\
&&+4I^2(FG-EH)+4J^2(BC-AD)+4IJ(AH-BG-CF+DE)\nonumber\\
&&+4K^2(DG-CH)+4L^2(BE-AF)+4KL(AH-BG+CF-DE)\nonumber\\
&&+4M^2(DF-BH)+4N^2(CE-AG)+4MN(AH+BG-CF-DE)\nonumber\\
&&+8\left(AJLN-BJLM-CJKN+DJKM-EILN+FILM+GIKN-HIKM\right).\nonumber
\eea
\begin{proof}
By brute force calculation.
\end{proof}
\end{lemma}
In \cite{hitchin2000}, Hitchin shows that for $\varphi\in\clo_-^+(\mu)$, one has
\[\delta|\varphi|^2=2\frac{\delta\varphi\wedge\hat\varphi}{\omega^3/3!}.\]
This translates to our language as
\[\delta Q(\varphi)=-4\frac{\delta\varphi\wedge F(\varphi)}{\omega^3/3!},\]
which holds for arbitrary $\varphi$ by its algebraicity.
In a local frame with the notation in Lemma \ref{bc}, we have
\[\frac{Q(\varphi)}{2}=-A\hat H+B\hat G+C\hat F-D\hat E+E\hat D-F\hat C-G\hat B+H\hat A-2I\hat J+2J\hat I-2K\hat L+2L\hat K-2M\hat N+2M\hat N.\]
Hitchin's result in local coordinates is equivalent to the following seeming wrongly-scaled formulae:
\bea
\frac{\pt Q}{\pt A}=-8\hat H,&\quad &\frac{\pt Q}{\pt H}=8\hat A,\nonumber\\
\frac{\pt Q}{\pt B}=8\hat G,&\quad &\frac{\pt Q}{\pt G}=-8\hat B,\nonumber\\
\frac{\pt Q}{\pt C}=8\hat F,&\quad &\frac{\pt Q}{\pt F}=-8\hat C,\nonumber\\
\frac{\pt Q}{\pt D}=-8\hat E,&\quad &\frac{\pt Q}{\pt E}=8\hat D,\nonumber\\
\frac{\pt Q}{\pt I}=-16\hat J,&\quad &\frac{\pt Q}{\pt J}=16\hat I,\nonumber\\
\frac{\pt Q}{\pt K}=-16\hat L,&\quad &\frac{\pt Q}{\pt L}=16\hat K,\nonumber\\
\frac{\pt Q}{\pt M}=-16\hat N,&\quad &\frac{\pt Q}{\pt N}=16\hat M.\nonumber
\eea
In fact, the scaling here is the consequence of the Euler's identity.

\begin{ex}
Now let us give a closer look to the nilmanifold considered in \cite[Example 5.2]{bartolomeis2006} and \cite[pp. 798-799]{fei2021b}, where the Lie algebra of the nilpotent Lie group is characterized by left-invariant 1-forms $\{e^1,\dots,e^6\}$ satisfying
\[\begin{split}&\ud e^1=\ud e^2=\ud e^3=\ud e^5=0,\\
&\ud e^4=e^{15},\quad \ud e^6=e^{13},\end{split}\]
and the symplectic form $\omega$ is
\[\omega=e^{12}+e^{34}+e^{56}.\]
We consider the Type IIA flow with $\varphi$ taking the form of (\ref{generalphi}). Notice that the kernel of the operator $\ud\Lambda_\omega\ud$ on left-invariant primitive 3-forms is 13-dimensional:
\bea
\ker\ud\Lambda_\omega\ud&=&\spann\left\{e^{135},e^{136},e^{145},e^{146},e^{235},e^{236},e^{245},e^1(e^{34}-e^{56}),e^2(e^{34}-e^{56}), \right.\nonumber\\
&&\left.e^3(e^{12}-e^{56}),e^4(e^{12}-e^{56}),e^5(e^{12}-e^{34}),e^6(e^{12}-e^{34})\right\},\nonumber\\
\ud\Lambda_\omega\ud(e^{246})&=&-2e^{135}.\nonumber
\eea
Therefore the Type IIA flow with ansatze (\ref{generalphi}) reduces to an ODE
\[\p_t A=4\hat H=-4AH^2+R,\]
where
\[R=4H(BG+CF+DE-2IJ-2KL-2MN)+8(DJ^2+FL^2+GN^2-DFG-2JLN),\]
and $B,C,\dots,M,N$ are constants depending only on initial data.
The solution of this ODE is 
\[A(t)=e^{-4H^2t}\left[\frac{R(e^{4H^2t}-1)}{4H^2}+A(0)\right].\]
When $H=0$, the above expression should be understood as $A(t)=A(0)+Rt$. We see that the flow always has long time existence and it converges if and only if $H\neq 0$ or $H=R=0$. In addition, $\varphi$ is integrable, namely $\ud$-closed, if and only if $H=J=L=N=0$, in which case $R=-8DFG$.

The divergent case is that $H=0$ and $R\neq 0$, so we have
\[\lim_{t\to\infty}\frac{\varphi(t)}{Rt}=e^{135}=:\varphi_{\infty}.\] 
In this case $\ker\varphi_\infty=\spann\{e_2,e_4,e_6\}$ defines a Lagrangian foliation on the nilmanifold. In fact, it gives rise to a Lagrangian fibration since the Lagrangian subspace $\spann\{e_2,e_4,e_6\}$ of the Lie algebra is an ideal and it corresponds to a normal Lagrangian subgroup of the nilpotent Lie group. This verifies the metrical picture that along the Type IIA flow, the fibers of this Lagrangian fibration collapses and the manifold after scaling converges to its base in the sense of Gromov-Hausdorff. 

When $H\neq 0$, we have 
\[\lim_{t\to\infty}A(t)=\frac{R}{4H^2}.\]
In other words, $\varphi(t)$ converges to a left-invariant 3-form with $\hat H=0$, which is a stationary points of the Type IIA flow. Straightforward calculation yields the following result:
\begin{prop}
Following the notation in Lemma \ref{bc}, we have
\begin{enumerate}
\item $\varphi$ is integrable if and only if $H=J=L=N=0$.
\item $\varphi$ is a stationary point of the Type IIA flow if and only if $\hat H=0$.
\item $\varphi$ is an integrable stationary point of the Type IIA flow if and only if $H=J=L=N=0$ and $\hat H=-2DFG=0$. In this case 
\[\frac{Q(\varphi)}{4}=2(BG+CF+DE)^2-(BG+CF+DE)^2+4(I^2FG+K^2DG+M^2DF)\]
can be of any sign. When $Q(\varphi)$ is negative, $\varphi$ belongs to the orbit $\clo_-^-(\mu)$ hence the associated metric $q(\omega,\varphi)$ must have signature $(2,4)$.
\item $\varphi$ is integrable and $F$-integrable if and only if $H=J=L=N=0$ and $DFG=DKG=DFM=IFG=0$. In this case
\[\frac{Q(\varphi)}{4}=2(BG+CF+DE)^2-(BG+CF+DE)^2\geq 0\]
since $DFG=0$. As a consequence, there exist integrable and $F$-integrable forms from any of the orbits $\clo_+(\mu)$, $\clo_0^\pm$, $\clo_1^\pm$, $\clo_3\cap\bigwedge_0^3V^*$ and $\clo_6=\{0\}$. In particular, there exist para-K\"ahler Calabi-Yau structures on the associated nilmanifolds.
\end{enumerate}
\end{prop}
\end{ex}

\begin{ex}
Next, let us investigate the solvmanifold in \cite[Theorem 4.5]{tomassini2008} and \cite[pp. 799-802]{fei2021b} in greater generality, where the Lie algebra of the 6-dimensional solvable Lie group is characterized by
\[\begin{split}&\ud e^1=-\lambda e^{15},\quad\ud e^2=\lambda e^{25},\quad\ud e^3=-\lambda e^{36},\\
&\ud e^4=\lambda e^{46},\quad \ud e^5=0,\quad \ud e^6=0,\end{split}\]
with $\lambda=\log\dfrac{3+\sqrt{5}}{2}$. In addition, we take the symplectic form to be $\omega=e^{12}+e^{34}+e^{56}$. For simplicity, we shall only consider the Type IIA flow on it with closed initial data. In view of (\ref{generalphi}) it means $A=B=\alpha$, $C=-D=\beta$, $E=-F=\gamma$, $G=H=-\delta$, and $I=J=K=L=0$. In other words, any left-invariant closed primitive $\varphi$ on this solvable Lie group is of the form
\[\varphi=\alpha(e^{135}+e^{136})+\beta(e^{145}-e^{146})+\gamma(e^{235}-e^{236})-\delta(e^{245}+e^{246})+(Me^5+Ne^6)(e^{12}-e^{34}).\]
By direct calculation, we know that a basis of $\ker\ud\Lambda_\omega\ud$ consists of $e^{135}+e^{136}$, $e^{145}-e^{146}$, $e^{235}-e^{236}$, $e^{245}+e^{246}$, $e^1(e^{34}-e^{56})$, $e^2(e^{34}-e^{56})$, $e^3(e^{12}-e^{56})$, $e^4(e^{12}-e^{56})$, $e^5(e^{12}-e^{34})$, $e^6(e^{12}-e^{34})$. In addition, we have
\[\begin{split}&\ud\Lambda_\omega\ud e^{135}=-\ud\Lambda_\omega\ud e^{136}=-\lambda^2(e^{135}+e^{136}),\\
&\ud\Lambda_\omega\ud e^{145}=\ud\Lambda_\omega\ud e^{146}=\lambda^2(e^{145}-e^{146}),\\
&\ud\Lambda_\omega\ud e^{235}=\ud\Lambda_\omega\ud e^{236}=\lambda^2(e^{235}-e^{236}),\\
&\ud\Lambda_\omega\ud e^{245}=-\ud\Lambda_\omega\ud e^{246}=-\lambda^2(e^{245}+e^{246}),\end{split}\]
and
\[\begin{split}&\hat A=4\alpha\beta\gamma+2\alpha MN-2\alpha M^2,\\
&\hat B=-4\alpha\beta\gamma-2\alpha MN+2\alpha N^2,\\
&\hat C=-4\alpha\beta\delta+2\beta MN+2\beta M^2,\\
&\hat D=-4\alpha\beta\delta+2\beta MN+2\beta N^2,\\
&\hat E=-4\alpha\gamma\delta+2\gamma MN+2\gamma M^2,\\
&\hat F=-4\alpha\gamma\delta+2\gamma MN+2\gamma N^2,\\
&\hat G=-4\beta\gamma\delta-2\delta MN+2\delta M^2,\\
&\hat H=4\beta\gamma\delta+2\delta MN-2\delta N^2,\end{split}\]
together with
\[\begin{split}&\hat I=\hat J=\hat K=\hat L=0,\\
&\hat M=-2\alpha\delta(M-N)+2\beta\gamma(M+N),\\
&\hat N=-2\alpha\delta(M-N)-2\beta\gamma(M+N).\end{split}\]
It follows that the Type IIA flow reduces to the ODE system
\bea
\begin{cases}\label{solvsystem}
\pt_t\alpha&=4\lambda^2\alpha(4\beta\gamma-(M-N)^2),\\
\pt_t\beta&=4\lambda^2\beta(4\alpha\delta-(M+N)^2),\\
\pt_t\gamma&=4\lambda^2\gamma(4\alpha\delta-(M+N)^2),\\
\pt_t\delta&=4\lambda^2\delta(4\beta\gamma-(M-N)^2),
\end{cases}
\eea
where $M$ and $N$ are constants depending on initial data. Moreover, we have that
\[\frac{Q(\varphi)}{4}=-16\alpha\beta\gamma\delta+4\alpha\delta(M-N)^2+4\beta\gamma(M+N)^2.\]

Like the previous example, we can classify all left-invariant integrable stationary points of the Type IIA flow on this solvable Lie group.
\begin{prop}
Following the notation in Lemma \ref{bc}, we have
\begin{enumerate}
\item $\varphi$ is integrable if and only if $A=B=\alpha$, $C=-D=\beta$, $E=-F=\gamma$, $G=H=-\delta$, and $I=J=K=L=0$.
\item $\varphi$ is an integrable stationary point of the Type IIA flow if and only if it is integrable and $F$-integrable. In this case, we have $A=B=\alpha$, $C=-D=\beta$, $E=-F=\gamma$, $G=H=-\delta$, $I=J=K=L=0$, and $\alpha(4\beta\gamma-(M-N)^2)=\delta(4\beta\gamma-(M-N)^2)=\beta(4\alpha\delta-(M+N)^2)=\gamma(4\alpha\delta-(M+N)^2)=0$. In particular, we know $Q(\varphi)=64\alpha\beta\gamma\delta\geq 0$.
\item If $\varphi$ is both integrable and $F$-integrable, then $\varphi$ lies in one of the following orbits $\clo_+(\mu)$, $\clo_1^\pm$, $\clo_3\cap\bigwedge_0^3V^*$ and $\clo_6=\{0\}$.
\end{enumerate}
\end{prop}
\begin{proof}
Part (a) and (b) follow directly from standard calculation. For Part (c), it is easy to find $\alpha$, $\beta$, $\gamma$, $\delta$, $M$ and $N$ such that
\bea
4\alpha\delta-(M+N)^2=4\beta\gamma-(M-N)^2=0,\label{simpcon}
\eea 
with $Q(\varphi)=64\alpha\beta\gamma\delta>0$. In this case we know that $\varphi$ belongs to the orbit $\clo_+(\mu)$ and it gives rise to a para-K\"ahler Calabi-Yau structures on the associated solvmanifolds. Therefore we only need to consider the other orbits for which we also have $Q(\varphi)=0$.

We shall consider the following two cases: (i) Suppose Equation (\ref{simpcon}) holds, then we have $Q(\varphi)=64\alpha\beta\gamma\delta=4(M-N)^2(M+N)^2=0$. Without loss of generality, we may assume that $(M-N)^2=4\beta\gamma=0$, which implies that $M=N$ and $\alpha\delta=M^2=MN=N^2$. It follows that $\hat A=\hat B=\dots=\hat H=\hat M=\hat N=0$ hence $F(\varphi)=0$. (ii) If Equation (\ref{simpcon}) fails, then without loss of generality we may assume that $4\alpha\delta-(M+N)^2\neq 0$. The conditions in (b) imply that $\beta=\gamma=0$ and $\alpha(M-N)^2=\delta(M-N)^2=0$. These equations also lead to $\hat A=\hat B=\dots=\hat H=\hat M=\hat N=0$ hence $F(\varphi)=0$. As a consequence, we always have $F(\varphi)=0$ hence the orbits $\clo_0^\pm$ are ruled out. It is not hard to construct integrable and $F$-integrable $\varphi$ belonging to the orbits $\clo_1^\pm$, $\clo_3\cap\bigwedge_0^3V^*$ and $\clo_6=\{0\}$ explicitly.
\end{proof}

Now let us turn to the ODE system (\ref{solvsystem}). For simplicity, we shall only consider the case where the initial data for $\varphi$ lies in the orbit $\clo_-^+(\mu)$ for certain $\mu>0$. This condition translates to that the matrix
\[\begin{bmatrix}2\alpha\beta & & \alpha(N-M) & \beta(M+N) &  & \\
& 2\gamma\delta & \gamma(M+N) & \delta(M-N) & & \\
\alpha(N-M) & \gamma(M+N) & 2\alpha\gamma & & &\\
\beta(M+N) & \delta(M-N) & & 2\beta\delta & &\\
& & & & \alpha\delta+\beta\gamma-M^2 & \alpha\delta-\beta\gamma-MN\\
& & & & \alpha\delta-\beta\gamma-MN & \alpha\delta+\beta\gamma-N^2
\end{bmatrix}\]
is positive definite, which by \cite{fei2021b} is preserved under the Type IIA flow. By Sylvester's criterion, the positivity of the above matrix is equivalent to the following system of inequalities
\bea
\begin{cases}&\alpha,\beta,\gamma,\delta \textrm{ are all positive or all negative},\\
&\alpha\delta+\beta\gamma>M^2,\quad \alpha\delta+\beta\gamma>N^2,\\
&\dfrac{Q(\varphi)}{16}=-4\alpha\beta\gamma\delta+\alpha\delta(M-N)^2+\beta\gamma(M+N)^2<0,
\end{cases}\label{initial}
\eea
which are all preserved under the Type IIA flow.

The behavior of the ODE system (\ref{solvsystem}) when $M=N=0$ is analyzed in \cite{fei2021b}. For the general case, we have:
\begin{prop}
For any value of $M$ and $N$, the ODE system (\ref{solvsystem}) has finite time singularity. Let $T<\infty$ be its maximal existence time, then the limit
\[\lim_{t\to T}\alpha(t)^{-1}\varphi(t)=\varphi_{\infty}\]
exists smoothly. Moreover, $\varphi_{\infty}$ defines a harmonic almost complex structure, as detailed in \cite{fei2021b}.
\end{prop}
\begin{proof}
It is obvious that $\alpha/\delta$ and $\beta/\gamma$ are positive constants along the flow.

Without loss of generality, we may assume that $\alpha$, $\beta$, $\gamma$, and $\delta$ are all positive, otherwise we may consider the evolution equation for $-\alpha$, $-\beta$, $-\gamma$, and $-\delta$ instead. Since (\ref{initial}) holds along the flow, we know that
\[4\alpha\beta\gamma\delta>\alpha\delta(M-N)^2,\quad 4\alpha\beta\gamma\delta>\beta\gamma(M+N)^2,\]
so the right hand side of (\ref{solvsystem}) are all positive. It follows that there exists a small positive number $c$ depending on the initial data such that $\p_tf\geq cf$ holds when $f=\alpha,\beta,\gamma$ or $\delta$, therefore all of $\alpha$, $\beta$, $\gamma$, and $\delta$ have at least exponential growth.

Now let $u=4\alpha\delta$ and $v=4\beta\gamma$. The pair $(u,v)$ satisfies
\bea
\begin{cases}
&\p_tu=2\lambda^2u(v-(M-N)^2),\\
&\p_tv=2\lambda^2v(u-(M+N)^2),
\end{cases}\label{uv}
\eea
satisfying
\bea
u,v>0\textrm{ and } uv>u(M-N)^2+v(M+N)^2.\label{mother}
\eea
In fact, the last two lines in the system of inequalities (\ref{initial}) can be derived from (\ref{mother}) as follows. First we deduce that $u>(M+N)^2$ and $v>(M-N)^2$. So we get
\[\left(\frac{u+v-2M^2-2N^2}{2}\right)\geq (u-(M+N)^2)(v-(M-N)^2)>(M^2-N^2)^2.\]
It follows that 
\[u+v>2M^2+2N^2+2|M^2-N^2|=4\max\{M^2,N^2\}.\]
Let $S=\max\{(M+N)^2,(M-N)^2\}$ and we shall compare the ODE system (\ref{uv}) with the following ODE system
\bea
\begin{cases}
&\p_tu=2\lambda^2u(v-S),\\
&\p_tv=2\lambda^2v(u-S).
\end{cases}\label{uvnormal}
\eea
By the ODE comparison theorem, if we can show that (\ref{uvnormal}) blows up in finite time, then (\ref{uv}) blows up in even shorter time. It turns out that (\ref{uvnormal}) can be solved explicitly as follows. By taking the difference of the two equations in (\ref{uvnormal}), we get
\[\p_t(u-v)+2\lambda^2S(u-v)=0.\]
Therefore there exists a constant $C_0=u_0-v_0$ such that
\[u-v=C_0e^{-2\lambda^2St}.\]
Plug it back in (\ref{uvnormal}), we get
\bea
\p_t u=2\lambda^2u(u-C_0e^{-2\lambda^2St}-S).\label{u}
\eea
Let $w=e^{2\lambda^2St}u$, then from (\ref{u}) we get
\[\frac{\p_tw}{w(w-C_0)}=2\lambda^2e^{-2\lambda^2St}.\]
It follows that
\[w=\frac{C_0}{1-\frac{w_0-C_0}{w_0}e^{-\frac{C_0}{S}(e^{-2\lambda^2St}-1)}}.\]
As a consequence, we see that $w$ blows up at finite time $T'$, where
\[\begin{split}T'&=-\frac{1}{2\lambda^2S}\log\left[1+\frac{S}{C_0}\log\frac{w_0-C_0}{w_0}\right]\\
&=-\frac{1}{2\lambda^2S}\log\left[1+\frac{S}{\alpha_0\delta_0-\beta_0\gamma_0}\log\left(1 -\frac{\alpha_0\delta_0-\beta_0\gamma_0}{\alpha_0\delta_0}\right)\right].\end{split}\]
Now we can turn back to the ODE system (\ref{uv}), for which we now know that it blows up at a finite time $T$ and we have an estimate $T<T'$. Without loss of generality, we may assume $v(t)$ blows up at time $T$, i.e. $\lim_{t\to T}v(t)=+\infty$. Now we shall show that $u(t)$ also blows up at timer $T$.

By direct computation, we have
\[\p_t\left(\frac{(u-v)^2}{u^2}\right)=\frac{2\lambda^2}{u^2}\left[(u-v)((M+N)^2v-(M-N)^2u)-(u-v)^2(v-(M-N)^2)\right].\]
Firstly, it is easy to show that for any $\epsilon>0$, there exists a positive constant, whose optimal value is $C_\epsilon=\dfrac{8M^2N^2}{\epsilon}-(M+N)^2$, such that
\[(u-v)((M+N)^2v-(M-N)^2u)\leq C_\epsilon(u-v)^2+\epsilon u^2.\]
As $v(t)$ is increasing and $\lim_{t\to T}v(t)=+\infty$, there exists a time $T_1<T$ such that for any $T_1\leq t < T$, we have
\[v(t)-(M-N)^2\geq \frac{v(t)}{2}+C_\epsilon.\]
Therefore on the time interval $[T_1,T)$, we have
\bea
\p_t\left(\frac{(u-v)^2}{u^2}\right)\leq\frac{2\lambda^2}{u^2}\left[\epsilon u^2-\frac{v}{2}(u-v)^2\right]\leq 2\lambda^2\epsilon.\label{est}
\eea
Integrating it from $T_1$ to $t$ we see that
\[\left(\frac{u(t)-v(t)}{u(t)}\right)^2\leq C^2\]
on $[T_1,T)$, where the constant $C$ depends on $u(T_1)$, $v(T_1)$ and $\epsilon$. Therefore we know that on $[T_1,T)$, we have the estimate
\[v(t)\leq (1+C)u(t),\]
hence $\lim_{t\to\infty}u(t)=+\infty$.

Notice that the evolution equation of $u(t)$ can be written as
\[\frac{1}{2\lambda^2}\frac{\p_tu}{t}=v-(M-N)^2.\]
Integrating it, we get
\[\int_0^tv(s)\ud s=t(M-N)^2+\frac{1}{2\lambda^2}\log\frac{u(t)}{u(0)}\]
for any $0\leq t< T$. As we have proved that $\lim_{t\to T}u(t)=+\infty$, we can conclude
\[\lim_{t\to T}\int_0^tv(s)\ud s=+\infty.\]
Write $r=\dfrac{(u-v)^2}{u^2}$, then the first half of (\ref{est}) can be rewritten as
\[\p_tr+\lambda^2vr\leq 2\lambda^2\epsilon.\]
It follows that on $[T_1,T)$ we have
\[\p_t\left(e^{\lambda^2\int_0^tv(s)\ud s}r(t)\right)\leq 2\lambda^2\epsilon e^{\lambda^2\int_0^tv(s)\ud s}\leq \frac{2\epsilon}{v(T_1)}\lambda^2v(t)e^{\lambda^2\int_0^tv(s)\ud s}.\]
Integrating both sides, we get
\[e^{\lambda^2\int_0^tv(s)\ud s}r(t)\leq \frac{2\epsilon}{v(T_1)}\left(e^{\lambda^2\int_0^tv(s)\ud s}-e^{\lambda^2\int_0^{T_1}v(s)\ud s}\right)+e^{\lambda^2\int_0^{T_1}v(s)\ud s}r(T_1).\]
Divide both sides by $e^{\lambda^2\int_0^tv(s)\ud s}$ and let $t\to T$, we get
\[\lim_{t\to T}r(t)\leq\frac{2\epsilon}{v(T_1)}\]
for any $\epsilon>0$, hence we have proved
\bea
\lim_{t\to T}\frac{(u(t)-v(t))^2}{u^2(t)}=0,\label{est2}
\eea
or equivalently
\bea
\lim_{t\to T}\frac{u(t)}{v(t)}=1.\label{est3}
\eea
As $u=4\alpha\delta$ and $v=4\beta\gamma$ with $\alpha/\delta$ and $\beta/\gamma$ are positive constants, the limit (\ref{est3}) implies that
\[\begin{split}
&\lim_{t\to T}\frac{\beta(t)}{\alpha(t)}=\beta_\infty>0,\\
&\lim_{t\to T}\frac{\gamma(t)}{\alpha(t)}=\gamma_\infty>0,\\
&\lim_{t\to T}\frac{\delta(t)}{\alpha(t)}=\delta_\infty>0.
\end{split}\]
Let $\alpha_\infty=1$. Since $M$ and $N$ are constants, we conclude that
\[\lim_{t\to T}\frac{\varphi(t)}{\alpha(t)}=\alpha_\infty(e^{135}+e^{136})+\beta_\infty(e^{145}-e^{146})+\gamma_\infty(e^{235}-e^{236})-\delta_\infty(e^{245}+e^{246})\]
exists smoothly. Moreover, the estimate (\ref{est2}) implies that
\[\alpha_\infty\delta_\infty-\beta_\infty\gamma_\infty=0,\]
so from \cite{fei2021b} we know that the almost complex structure associated to $\varphi_\infty$ is harmonic.
\end{proof}
\begin{rmk}
The limit (\ref{est2}) can actually be strengthen to
\[\lim_{t\to T}\frac{(u(t)-v(t))^2}{u(t)}=0.\]
\end{rmk}
\end{ex}

\bibliographystyle{alpha}

\bibliography{C:/Users/benja/Dropbox/Documents/Source}

\end{document}